\documentclass[letterpaper, 11pt]{article} 
\usepackage{graphics,graphicx}
\usepackage{parskip}
\usepackage{amsmath}
\usepackage{amsthm}
\usepackage{mathtools}
\usepackage[english]{babel}
\usepackage[utf8]{inputenc}
\usepackage[title]{appendix}
\usepackage{amssymb} 
\usepackage[font=footnotesize,labelfont=small]{caption}
\RequirePackage{geometry}
\newtheorem{theorem}{Theorem}[section]
\newtheorem{lemma}[theorem]{Lemma}
\newtheorem{proposition}[theorem]{Proposition}
\newtheorem{corollary}[theorem]{Corollary}

\newtheorem{remark}{Remark}

\makeatletter
\DeclareRobustCommand\widecheck[1]{{\mathpalette\@widecheck{#1}}}
\def\@widecheck#1#2{%
    \setbox\z@\hbox{\m@th$#1#2$}%
    \setbox\tw@\hbox{\m@th$#1%
      \widehat{%
          \vrule\@width\z@\@height\ht\z@
          \vrule\@height\z@\@width\wd\z@}$}%
    \dp\tw@-\ht\z@
    \@tempdima\ht\z@ \advance\@tempdima2\ht\tw@ \divide\@tempdima\thr@@
    \setbox\tw@\hbox{%
      \raise\@tempdima\hbox{\scalebox{1}[-1]{\lower\@tempdima\box
\tw@}}}%
    {\ooalign{\box\tw@ \cr \box\z@}}}
\makeatother

\newcommand{\ep}{\epsilon}
\newcommand{\om}{\omega}

\DeclarePairedDelimiter\floor{\lfloor}{\rfloor}

\DeclareMathOperator{\Res}{Res}

\DeclareMathOperator{\sinc}{sinc}

\newcommand{\ds}{\displaystyle}

\newcommand{\be}{\begin{equation}}
\newcommand{\ee}{\end{equation}}
\newcommand{\bes}{\begin{equation*}}
\newcommand{\ees}{\end{equation*}}
\newcommand{\mand}{\quad \text{and}\quad}

\newcommand{\R}{{\bf{R}}}
\newcommand{\E}{{\mathbb{E}}}

\newcommand{\Z}{{\bf{Z}}}

\renewcommand{\L}{{\mathcal{L}}}

\newcommand{\A}{{\mathcal{A}}}

\renewcommand{\S}{{\mathcal{S}}}

\renewcommand{\tilde}{\widetilde}

\renewcommand{\check}{\widecheck}

\newcommand{\norm}[1]{\|#1\|}

\numberwithin{equation}{section}
 \usepackage[pagewise]{lineno}

\title{Using Random Walks to Establish Wavelike Behavior in an FPUT System with Random Coefficients }
\author{Joshua A. McGinnis, J. Douglas Wright}

\date{\today}
\begin{document}
\maketitle
\begin{abstract}
    We consider a linear Fermi-Pasta-Ulam-Tsingou lattice with random spatially varying material coefficients. Using the methods of stochastic homogenization we show that solutions with long wave initial data converge in an appropriate sense to solutions of a wave equation. The convergence is strong and both almost sure and in expectation, but the rate is quite slow. The technique combines energy estimates with powerful classical results about random walks, specifically the law of the iterated logarithm.
\end{abstract}

\section{Introduction}
\label{sec:Intro}
We prove an almost sure convergence result for solutions of the following one-dimensional random polymer linear Fermi-Pasta-Ulam-Tsingou (FPUT) lattice in the long wave limit:
\begin{align}\label{RFPUT}
m(j)\ddot{u}(j) = k(j)\left[u(j+1)-u(j)\right]-k(j-1)\left[u(j)-u(j-1)\right].
\end{align}
Here $j \in \Z$, $u=u(j,t) \in \R$ and $t \in \R$.
We choose the coefficients $m(j)$  (which we refer to as ``the masses'') to be independent and identically distributed (i.i.d.) random variables contained almost surely in some intervals $[a_m,b_m] \subset \R^+$ 
with standard deviation $\sigma_m$. We similarly take the coefficients $1/k(j)$ (``the springs'')
to be i.i.d. with support in $[a_k,b_k] \subset \R^+$ and deviation $\sigma_k$.
This system is well-understood when these coefficients are either constant or periodic with respect to $j$ 
\cite{Mielke}, but for the random problem most of what is known is formal or numerical \cite{Porter,Okada}.

For initial conditions whose wavelength is  $O(1/\ep)$, with $\ep \in (0,1)$, we prove that the $\ell^2$ norm of the difference between true solutions and appropriately scaled solutions to the wave equation is at most $O\left(\sqrt{\log\log(1/\epsilon)}\right)$ 
for times of $O(1/\ep)$ for almost every realization. While such an absolute error diverges
as $\ep \to 0^+$, it happens that this is 
 enough to establish an almost sure convergence result within the ``coarse-graining'' setting 
 used in \cite{Mielke} to study the (multi-dimensional) periodic problem. 
In addition to the almost sure convergence, we are able to provide estimates on the mean of the error in terms of $\sigma_m$ and $\sigma_k$ and prove convergence in mean.

The articles \cite{Chirilius-Bruckner,Wright} study the nonlinear FPUT lattice with periodic coefficients.
These show that soliton-like solutions exist for very large time scales using Korteweg-de Vries (KdV) approximations.
The authors of \cite{Wright} used the so called multiscale method of homogenization, a 
by-now classical tool with a long history in PDE for deriving effective equations, see \cite{Cioranescu}. In this paper, we carry out a very similar approach in deriving and proving the results; however, our expansions only result in an effective wave equation, not the KdV equation. In our setting, since the coefficients are random, it is necessary to average over the entire lattice. The law of large numbers implies this average is equal to the expectation, so the speed of the approximate solution depends on the expectation of the random variables. The probability theory hinges upon classical but extremely powerful asymptotic analysis of random walks, namely the law of iterated logarithms, as well as basic martingale theory. 

We denote a doubly infinite sequence $\{x(j)\}_{j \in \Z}$ by $x$. Let $S^{\pm}$ be the shift operators which act on sequences $f = \{f(j)\}$ as 
\begin{equation*}(S^{\pm}f)(j) \coloneqq f(j\pm1), \end{equation*}
and the operators $\delta^+$ and $\delta^-$, the left and right difference operators, are 
\begin{equation*} \begin{aligned}(\delta^+f)(j) &:=f(j+1)-f(j) \\ (\delta^-f)(j) &:= f(j) -f(j-1). \end{aligned}  \end{equation*}
Defining 
\begin{equation*}\begin{aligned} r& \coloneqq \delta^+u \\ p& \coloneqq \dot{u}, \end{aligned}\end{equation*}
we convert our second order equation (1.1) to the system 
\begin{equation}\label{FORFPUT}
\begin{aligned} 
    \dot{r} &= \delta^+p  \\ 
    \dot{p} &=\dfrac{1}{m}\delta^-(k r).
\end{aligned} 
\end{equation}
For the remainder of the paper, we work with \eqref{FORFPUT}. 

Here is the  idea of our ultimate result. Suppose that the initial conditions for \eqref{FORFPUT}
have the following long wave form:
$$
r(j,0) =\Phi(\ep j)/k(j) \mand p(j,0) = \Psi(\ep j)
$$
where $\Phi,\Psi:\R \to \R$ are suitably smooth, of somewhat rapid decay, and 
$\ep \in (0,1)$.
Then
the solution $p$ of \eqref{FORFPUT} has
$$
\L [p](X/\ep,\tau/\ep) \longrightarrow P_0(X,\tau)
$$
as $\ep \to 0^+$
where $P_0$ solves the wave equation
$\partial_\tau^2 P_0 = c^2 \partial_X^2 P_0$. The operator $\L$ 
interpolates the sequence $p$ into a function on $\R$. It is defined below, as is the wave speed $c$.
The convergence is strong in $L^2(\R)$ and is both almost sure and in expectation. A similar convergence
holds for $r$.

The paper is organized as follows. We carry out the multiscale expansion in Section \ref{homog section} and derive effective equations and approximate solutions. In Section \ref{prob est sec} we dive into the analysis of various smooth, rapidly decaying functions which are sampled at integers and multiplied componentwise by random walks. These estimates are necessary to control the error and here is where most the probability theory is needed. In Section \ref{Err est sec} we provide the rigorous estimates of the error. We introduce coarse-graining and prove the convergence results in Section \ref{Coarse grain sec}. In Section \ref{conc} we provide numerical simulations as evidence that our estimates are good ones i.e. they are not vast overestimates.


\section{Homogenization and derivation of the effect wave equation}\label{homog section}
In this section we homogenize the equation following closely what is done in \cite{Wright}. First, we define ``residuals'', which quantify how close some function is to a true solution. For any functions 
$\tilde{r}(j,t)$ and $\tilde{p}(j,t)$ put
\begin{equation} \label{res def}
\begin{aligned} \text{Res}_1(\tilde{r}, \tilde{p}) &:= \delta^+\tilde{p} -\partial_t\tilde{r}
\\ \ \text{Res}_2(\tilde{r}, \tilde{p}) &:=\dfrac{1}{m}\delta^-(k\tilde{r})-\partial_t\tilde{p}.\end{aligned} \end{equation} 
We look for approximate long wave solutions of the form
\begin{equation}\label{ansatz1}
\begin{aligned} \tilde{r}(j,t) &= \tilde{r}_\epsilon(j,t) : = R(j,\epsilon j,\epsilon t) \\ \ \tilde{p}(j,t) &= \tilde{p}_\epsilon (j,t): = P(j,\epsilon j,\epsilon t) , \end{aligned} \end{equation} where $R=R(j,X,\tau)$ and $P=P(j,X,\tau)$ are maps 
\begin{equation*}\Z \times \R \times \R \to \R. \end{equation*} 
 In the periodic-coefficient problem studied in \cite{Wright}, it was necessary to assume that these functions are periodic in the $\Z$ slot, but this needs to be exchanged in the random case. Here,  we make a ``sublinear growth'' assumption that makes averaging possible:
 \begin{equation}\label{limits}
 \lim\limits_{|j| \to \infty} R(j,X,\tau)/j   = \lim\limits_{|j| \to \infty} P(j,X,\tau)/j =0.   \end{equation} 
 The following lemma is crucial to the derivation of the effective equations. 

\begin{lemma}\label{solve}
There exists an $f = \left\{ f(j)\right\}_{j \in \Z}$ satisfying both 
\begin{equation}\label{fast decay}
\lim\limits_{|j| \to \infty} f(j)/j \to 0 \end{equation}
and 
\begin{equation*}(\delta^\pm f )(j)= g(j) \end{equation*}
if and only if 
\begin{equation*} \lim\limits_{N \to \infty} \dfrac{1}{N} \sum\limits_{i = 0}^{N-1} g(i) = 
\lim\limits_{N \to \infty} \dfrac{1}{N} \sum\limits_{i = 1}^{N} g(-i) = 0.\end{equation*} 
\end{lemma}

\begin{proof}
 We only give a proof for ``$\delta^+$''. $\Rightarrow$ Since $g(j) =f(j+1)-f(j)$ we get 
  \begin{equation*}\lim_{N \to \infty} \dfrac{1}{N}\sum\limits_{i=0}^{N-1}g(i) = \lim_{N \to \infty} \dfrac{1}{N}\sum\limits_{i=0}^{N-1}[f(i+1)-f(i)]=\lim_{N \to \infty}\dfrac{f(N)-f(0)}{N} =0\end{equation*} by assumption
  \eqref{fast decay}. Proof of the second such equality is analogous. 
  
$\Leftarrow$ If we choose $f(0)=0$, it is readily checked that
\begin{equation} \label{solve for chi} f(j)=\sum_{k=0}^{j-1}g(i) \mand f(-j)=-\sum_{k=1}^{j}g(-i) \end{equation}
 for $j>0$ solves $\delta^+f=g$.
Then 
  \begin{equation*}\lim\limits_{j \to \infty} \dfrac{f(j)}{j} = \lim\limits_{j \to \infty} \dfrac{\sum\limits_{i=0}^{j-1}g(i)}{j} =0.\end{equation*} It is likewise seen that \begin{equation*}\lim_{j \to \infty}\dfrac{f(-j)}{-j}=0 \end{equation*} by using the formula for $f(-j)$.   
\end{proof}
 
 Now we continue with the homogenization procedure. We must understand how $\delta^{\pm}$ act on functions of the type \eqref{ansatz1}.  The following expansions are found in \cite{Wright}.
 If $u(j) = U(j, \epsilon j),$ then 
 \begin{equation*}\delta^{\pm}u(j)= \sum\limits_{n \geq 0}\epsilon^n \delta_n^{\pm}U \end{equation*}
 where \begin{equation*} \delta_0^\pm := \check{\delta}^{\pm} \ \text{and} \ \delta_n^{\pm}:=\dfrac{(\pm 1)^{n+1}}{n!}\check{S}^{\pm} \partial_X^n.\end{equation*} Here $\check{\delta}^{\pm} $ and $\check{S}^{\pm}$ act only on the first slot; they are analogous to partial derivatives with respect to $j$. Precisely,\begin{equation*} \begin{aligned} 
 \check{S}^{+}(U)(j,X) &:= U(j+1,X), \\
 \check{S}^{-}(U)(j,X) &:= U(j-1,X), \\
 \check{\delta}^{+}(U)(j,X) &:= U(j+1,X) -U(j,X),  \ \\
 \check{\delta}^{-}(U)(j,X) &:= U(j,X) -U(j-1,X).
 \end{aligned} 
 \end{equation*} 
Let 
 \begin{equation*} (E_M^{\pm}u)(j) : = (\delta^\pm u)(j) - \sum\limits_{n=0}^M\epsilon^n(\delta^\pm_n U)(j,\epsilon j) \end{equation*}
 be the error made by truncating the series expansion of $\delta^\pm u$ after $M$ terms. Thus the lowest power of $\epsilon$ we see in the error term is  $\epsilon^{M+1}$. 
 
We further assume that our approximate solutions $R$ and $P$ themselves have expansions in $\epsilon$:
\be\label{expano}
R(j,X,\tau) = R_0(j,X,\tau)+\ep R_1(j,X,\tau) \mand
P(j,X,\tau) = P_0(j,X,\tau)+\ep P_1(j,X,\tau).
\ee
Of course $R_i(j,X,\tau)$  and $P_i(j,X,\tau)$ meet \eqref{fast decay}. 
Using the above expansion, we directly compute $\text{Res}_1(\tilde{r}_\epsilon,\tilde{p}_\epsilon)$:
\begin{equation}\label{res1v1}
\begin{aligned}\text{Res}_1(\tilde{r}_\epsilon, \tilde{p}_\epsilon ) &=   \delta_0^+P_0 +\epsilon \delta^+_1P_0 +E_1^+(P_0)  \\ &+\epsilon \delta_0^+P_1 +\epsilon^2 \delta^+_1P_1 +\epsilon E_1^+(P_1) \\ &-\epsilon \partial_\tau R_0 - \epsilon^2 \partial_\tau R_1.  \end{aligned}  \end{equation} 
Here we have used the expansion for $\delta^+$. Similarly 
\begin{equation}\label{res2v1}\begin{aligned} \text{Res}_2(\tilde{r}_\epsilon, \tilde{p}_\epsilon ) &=  \frac{1}{m}(\delta_0^-kR_0 +\epsilon \delta^-_1kR_0 +E_1^-(kR_0)  \\ &+\epsilon \delta_0^-kR_1 +\epsilon^2\delta^-_1kR_1 +\epsilon E_1^-(kR_1) \\  &-\epsilon \partial_\tau P_0  - \epsilon^2 \partial_\tau P_1).\end{aligned}\end{equation} 

Next set 
\be\label{this is Q}
Q_i \coloneqq kR_i.
\ee We choose $P_0, P_1, Q_0$ and $Q_1$ so that the $O(1)$ and $O(\epsilon)$ terms in \eqref{res1v1} and \eqref{res2v1}
vanish.  We get
\begin{equation}\label{o1}\begin{aligned}
\check{\delta}^+P_0 &=0 \\ \dfrac{1}{m}\check{\delta}^-Q_0&=0 
\end{aligned} \tag{$O(1)$}\end{equation}
and
\begin{equation}\label{oep}\begin{aligned}
\check{\delta}^+P_1 &= \dfrac{1}{k}\partial_{\tau}Q_0-S^{+}\partial_XP_0 \\ \check{\delta}^-Q_1 &= m \partial_{\tau} P_0-S^{-1}\partial_XQ_0.
\end{aligned} \tag{$O(\epsilon)$} \end{equation}

From \eqref{o1} we learn that $P_0$ and $Q_0$ do not depend on $j$, i.e. 
\begin{equation}\label{noj}
P_0 (j,X,\tau)= \bar{P}_0(X,\tau) \ \text{and} \ Q_0 (j,X,\tau)= \bar{Q}_0(X,\tau). \end{equation} 
If there are to be solutions $P_1$ and $Q_1$ or \eqref{oep} which satsify \eqref{fast decay},
 Lemma 2.1 tells us we must have \begin{equation}\begin{aligned}\lim_{N \to \infty}
\dfrac{1}{N}\sum_{j=0}^{N-1}\left[\frac{1}{k(j)}\partial_{\tau}\bar{Q}_0-\partial_X\bar{P}_0\right] =\lim_{N \to \infty}
\dfrac{1}{N}\sum_{j=1}^{N}\left[\frac{1}{k(-j)}\partial_{\tau}\bar{Q}_0-\partial_X\bar{P}_0\right] =0 \\ \lim_{N \to \infty}
\dfrac{1}{N}\sum_{j=0}^{N-1} \left[m(j)\partial_{\tau} \bar{P}_0-\partial_X\bar{Q}_0\right]=
\lim_{N \to \infty}
\dfrac{1}{N}\sum_{j=1}^{N} \left[m(-j)\partial_{\tau} \bar{P}_0-\partial_X\bar{Q}_0\right]=0  .\end{aligned}\end{equation} 
Since $\bar{P}_0$ and $\bar{Q}_0$ do not depend upon $j$ these can be rewritten as
$$
\left[\lim_{N \to \infty} \dfrac{1}{N}\sum_{j=0}^{N-1}\frac{1}{k(j)} \right] \partial_{\tau}\bar{Q}_0 = 
\lim_{N \to \infty}\left[
\dfrac{1}{N}\sum_{j=1}^{N} \frac{1}{k(-j)}\right] \partial_{\tau}\bar{Q}_0=\partial_X \bar{P}_0
$$
and
$$
\left[\lim_{N \to \infty}
\dfrac{1}{N}\sum_{j=0}^{N-1} m(j)\right] \partial_\tau \bar{P}_0 = 
\left[\lim_{N \to \infty}
\dfrac{1}{N}\sum_{j=1}^{N} m(-j)\right] \partial_\tau \bar{P}_0 = \partial_X \bar{Q}_0.
$$

The law  of large numbers tells us that \begin{equation}\lim_{N \to \infty}\dfrac{1}{N}\sum_{j= 0}^{N-1}m(j)=\lim_{N \to \infty}\dfrac{1}{N}\sum_{j= 1}^{N}m(-j)=\E[m]=:\bar{m}\end{equation}
and
\begin{equation}
\lim_{N \to \infty}\frac{1}{N}\sum_{j=0}^{N-1} \frac{1}{k(j)} =\lim_{N \to \infty}\frac{1}{N}\sum_{j=0}^{N-1} \frac{1}{k(-j)}=\E\left[\frac{1}{k}\right]
=:\frac{1}{\tilde{k}
}
\end{equation} almost surely, since $m$ and $k$ are sequences of i.i.d. random variables. 
To be clear $\E[\cdot]$ is the expectation of a random variable. 
And so we find that
\begin{equation}\label{wave1}
\begin{aligned}   \partial_{\tau} \bar{Q}_0 &= \tilde{k}\partial_{X}\bar{P}_0  \\ \partial_{\tau}\bar{P}_0 &= \dfrac{1}{\bar{m}}\partial_{X}\bar{Q}_0. \end{aligned}  \end{equation} From this, out pops the effective wave equation
\begin{equation*}\partial^2_{\tau} \bar{Q}_0 = c^2\partial^2_X\bar{Q}_0\end{equation*}
with wave-speed
$$c:= \sqrt{{\tilde{k}/\bar{m}}}.$$
 
We can use d'Alemberts formula to get $\bar{Q}_0$ and subsequently find $\bar{P}_0$ from its relation to $\bar{Q}_0$:
\begin{equation}\label{dalembert}\begin{aligned}
\bar{Q}_0(X,\tau) &= A(X-c\tau)+B(X+c\tau) \\
 \bar{P}_0(X,\tau) &=\frac{1}{\sqrt{\tilde{k}\bar{m}}}(-A(X-c\tau) + B(X+c\tau)). \end{aligned}\end{equation} 
The functions  $A$ and $B$ will ultimately be determined by the initial conditions
for \eqref{FORFPUT} in a fashion that is consistent with \eqref{ansatz1}.

At this point we have computed the effective wave equation but  we must also determine 
the full form of ${P}_1$ and ${Q}_1$. Using \eqref{noj} and \eqref{wave1} in \eqref{oep} we get 
\begin{equation*}\begin{aligned} \check{\delta}^+P_1 &= \left(\dfrac{\tilde{k}}{k}-1\right)\partial_X \bar{P}_0 \\ \check{\delta}^- Q_1 &= \left(\dfrac{m}{\bar{m}}-1\right)\partial_X\bar{Q}_0.\end{aligned}\end{equation*}
Define $\chi_m$ and $\chi_k $ as the solutions to
\begin{equation*}\begin{aligned} \delta^+ \chi_k =\dfrac{\tilde{k}}{k}-1
\ \text{and} \
\delta^-\chi_m = \dfrac{m}{\bar{m}}-1. 
\end{aligned}\end{equation*}
Using formula \eqref{solve for chi} in Lemma \ref{solve}, we can solve explicitly for $\chi_k$ and $\chi_m.$ 
They are
\be\label{chis}\begin{aligned} 
\chi_k(j)&=\sum_{i=0}^{j-1}\left[\dfrac{\tilde{k}}{k(i)}-1\right] \ \text{and}\ \chi_k(-j)=\sum_{i=1}^{j}\left[1-\dfrac{\tilde{k}}{k(i)}\right] \\ \chi_m(j)&=\sum_{i=0}^{j-1}\left[\dfrac{m(i)}{\bar{m}}-1\right] \ \text{and}\ \chi_k(-j)=\sum_{i=1}^{j}\left[1-\dfrac{m(i)}{\bar{m}}\right].
\end{aligned}\ee
Observe that $\ds \dfrac{\tilde{k}}{k}-1$ and  $\dfrac{m}{\bar{m}}-1$ are mean zero random variables
and as such $\chi_k$ and $\chi_m$ are classical random walks.
 The expression for $Q_1$ and $P_1$ can be given in terms of $\chi_k$ and $\chi_m$:
\begin{equation}\label{P1Q1}
\begin{aligned} 
Q_1(j,X,\tau) &= \chi_m(j)\partial_{X}\bar{Q}_0(X,\tau)\\
P_1(j,X,\tau) &= \chi_k(j)\partial_{X}\bar{P}_0(X,\tau).
\end{aligned} \end{equation}
We need to know estimates for the norm of $P_1$ and $Q_1$ so that we can
estimate the residuals. Results are given in the next section.
Here is an important preview of what we find: the growth rates for random walks ultimately imply
that the terms $\ep P_1$ and $\ep R_1$ in \eqref{expano} are, despite appearances, not actually $O(\ep)$.\
This in turn implies that the residuals are not as small as their formal derivation (namely $O(\ep^2)$) would lead one to believe.
This is the main technical complication in this article and the key difference between the random problem we study here and the periodic or constant coefficient problems studied in \cite{Mielke}.

Before moving on, we now spell out our long wave approximation in detail. Putting together
\eqref{ansatz1},
\eqref{expano},
 \eqref{this is Q},
 \eqref{noj},
 \eqref{dalembert} and
 \eqref{P1Q1}
 we see that 
\be\label{ansatz2}
\begin{split}
\tilde{r}_\ep(j,t) &= 
{1 \over k(j)} \left(A(\ep(j-ct)) + B (\ep(j+ct))\right) + \ep{\chi_m(j) \over k(j)} \left(A'(\ep(j-ct)) + B' (\ep(j+ct)) \right)\\
\tilde{p}_\ep(j,t) &= 
{1 \over \sqrt{\tilde{k}\bar{m}} }\left(-A(\ep(j-ct)) + B (\ep(j+ct))\right) + \ep{\chi_k(j) \over \sqrt{\tilde{k}\bar{m}}} \left(-A'(\ep(j-ct)) + B' (\ep(j+ct)) \right).
\end{split}
\ee


\section{Probabilistic estimates}
\label{prob est sec}
In this section we provide tools which will allow us to compute the 
$\ell^{2}$ norms of the residuals for all $|t| \leq T_0/\epsilon$. The first subsection deals with almost sure and realization dependent estimates by making use of the the law of iterated logarithms (LIL). The second subsection provides estimates on the expectation of the norms using martingale inequalities.

\subsection{Almost Sure Estimates}
One can find the statement of the LIL in \cite{Durret} and more details can be found in \cite{Feller}. Here we present the theorem in a form convenient to us. 

\begin{theorem} (The Law of Iterated Logarithms)
\label{LIL}
Suppose $y(j)$ ($j \in \Z$) are i.i.d random variables with mean zero and $\E[y^2]=\sigma^2$. Define the (two-sided) random walk $\chi$
via 
\begin{equation}\label{rando}
\chi(j)\coloneqq\sum_{i=0}^{j-1}y(i) \mand \chi(-j)\coloneqq\sum_{i=1}^jy(-i)
\end{equation} 
for $j > 0$ and $\chi(0) = 0$.

Then \begin{equation*}\limsup_{|j| \to \pm \infty} \frac{\pm\chi(j)}{\sqrt{2|j| \log \log( |j|)}} \overset{a.s.}{=} \sigma .\end{equation*}
\end{theorem}

The LIL is an extremely sharp 
description of a random walk. It says that, with a probability of one, the magnitude of 
$\chi(j)$ exceeds the curve $\sigma\sqrt{2j\log\log (j)}$ (by any fixed amount) only a finite number of times 
but comes arbitrarily near it an infinite number of times. Here is how we use the LIL:
\begin{corollary}\label{the core}
For almost every realization of $\{k(j)\}$ and $\{m(j)\}$ there is a finite positive constant
$
C_\om = C_\om(k,m)
$
for which
$$
|\chi_k(j)| + |\chi_m(j)| \le C_\om \sqrt{|j| \log \log(|j|+e)}
$$
for all $j \in \Z$.
\end{corollary}

\begin{remark} The constant $C_\om$ is almost surely finite by the LIL, but it may be extremely large.
There is no way to determine its magnitude except in very special circumstances.
Note, however, it does not depend on $\ep$.
\end{remark}

\begin{remark} In this paper, we use a small modification of the usual ``big $C$'' notation.
If a constant in an estimate depends on the particular realization of the coefficients
we mark it as ``$C_\om$.'' 
If it does not, we omit the subscript $\omega$.
All such constants $C_\om$
are always almost surely finite.
No such constants will ever depend on $\ep$.
\end{remark}

\begin{proof}
The LIL implies that for almost every realization of $\{k(j)\}$ there is a natural number $N_k$ such that
$$
|\chi_k(j)| \le 2 \tilde{k}\sigma_k \sqrt{2|j| \log \log(|j|)} 
$$
when $|j| \ge N_k$. Then put 
$$
C_k:=\max\left\{ {2\tilde{k}\sigma_k  \sqrt{2}},
\max_{0<|j| \le N_k} { \chi_k(j) \over \sqrt{|j| \log \log(|j|)}} \right\}.
$$
It follows that $|\chi_k(j)| \le C_k  \sqrt{|j| \log \log(|j|)} \le  C_k  \sqrt{|j| \log \log(|j|+e)}$
for all $j$. The same argument can be used to estimate $\chi_m$.
\end{proof}

Given the growth rate in the LIL, we introduce a new norm fashioned to absorb it:
$$
\| F \|_{H^s_{LIL}} := \sum_{i = 0}^s\| (1+|\cdot|\log\log(|\cdot|+e))^{1/2}F^{(i)}\|_{L^2}.
$$
The space $H^s_{LIL}$ will be the completion of $L^2$ with respect to this norm. 
Similarly, we also introduce
$$
\| F \|_{H^s_{sr}} := \sum_{i = 0}^s \| (1+|\cdot|)^{1/2}F^{(i)}\|_{L^2}
$$
and the space $H^s_{sr}$. Note that $\|F\|_{H^s} \le \|F\|_{H^s_{sr}} \le \|F\|_{H^s_{LIL}}$
where $H^s$ is the usual $L^2$-based Sobolev space of functions $\R\to \R$ which are weakly $s$-times
differentiable. 


Now we unveil the two main estimates we need to provide almost sure control of the residuals.
\begin{lemma}\label{Big Lemma}
For any $T_0>0$
and
almost every realization of $\{k(j)\}$ and $\{m(j)\}$ there is a finite positive constant
$
C_\om = C_\om\left(k,m,T_0\right)
$
for which $\ep \in (0,1/2)$ implies
\begin{equation}\label{big est 1}
\sup_{|t| \leq T_0/\epsilon} \norm{\chi(\cdot)F(\epsilon ( \cdot -  c t))}_{\ell^2} \leq C_{\om}\epsilon^{-1} \sqrt{\log\log(\epsilon^{-1})}\|F\|_{H^1_{LIL}}
\end{equation}
and
\begin{equation}\label{big est 2}
\sup_{|t| \leq T_0/\epsilon} \norm{\chi(\cdot)\delta^\pm F(\epsilon( \cdot -  ct))}_{\ell^2} \leq C_\om \sqrt{\log\log(\epsilon^{-1})}\|F\|_{H^2_{LIL}}. \end{equation}
In the above $\chi$ is either $\chi_k$ or $\chi_m$.
\end{lemma}

To prove these we need some calculus estimates.
\begin{lemma} \label{calc lemma}
For all $\epsilon \in (0,1)$, and $a,b \in \R$
\begin{equation*}|a+b|\log\log(|a+b|+e) \leq |a|\log\log(2|a|+e)+|b|\log\log(2|b|+e)\end{equation*}
and 
\begin{equation*}\log\log(|x|+ e) \leq \log(2\log(\epsilon|x|+e))+\log\log(\epsilon^{-1}+e).\end{equation*}
\end{lemma}

\begin{proof}
The first inequality follows from the fact that $|x|\log\log(|x|+e)$ is a convex function.

We will show the second inequality in two steps. First we show that \begin{equation*}\log\log(|x|+ e)  \leq \log(2 \log( \epsilon|x|+\epsilon^{-1}+e)). \end{equation*}
Since $\log$ is monotonic, this inequality follows from
\begin{equation*}|x|+ e\leq (\epsilon|x|+\epsilon^{-1}+e)^2,\end{equation*} which is trivial.

Now we show that \begin{equation*}\log(2\log(\epsilon|x| +\epsilon^{-1}+e)) \leq \log(2\log(\epsilon|x|+e))+\log\log(\epsilon^{-1}+e). \end{equation*}
Note that at $x=0$, equality holds. For $x\geq 0$ we have that 
\begin{equation*}\frac{d}{dx}\log(\epsilon x+\epsilon^{-1}+e)=\frac{\epsilon }{\epsilon x +\epsilon^{-1}+e}\end{equation*} and 
\begin{equation*}\frac{d}{dx}\log(\epsilon|x|+e)\log(\epsilon^{-1}+e)=\frac{\epsilon \log(\epsilon^{-1}+e)}{\epsilon x+e}. \end{equation*}
Since 
\begin{equation*}
\frac{\epsilon }{\epsilon x +\epsilon^{-1}+e} 
\leq 
{\ep \over \ep x + e}\le
\frac{\epsilon \log(\epsilon^{-1}+e)}{\epsilon x+e},\end{equation*} 
we see that $\log(\epsilon x+\epsilon^{-1}+e) $ grows more slowly than $\log(\epsilon x+e)\log(\epsilon^{-1}+e)$. Since both functions are even, we get by symmetry that \begin{equation*}2\log(\epsilon|x|+\epsilon^{-1}+e) \leq 2\log(\epsilon|x|+e)\log(\epsilon^{-1}+e). \end{equation*} Taking $\log$ of both sides, we get the desired result. 
\end{proof}

Now we can prove our key estimates.
\begin{proof} (Lemma \ref{Big Lemma})
Take $\chi$ to be $\chi_k$ or $\chi_m$ and fix $T_0>0$.
Using Corollary \ref{the core} 
\begin{equation*}\begin{aligned}\norm{ \chi(\cdot)F(\epsilon (\cdot -  c t))}_{\ell^2} &=  \left(\sum_{j \in \Z} \chi(j)^2F(\epsilon (j - c t))^2\right)^{1/2}\\ 
&\leq C_\om \left(\sum_{j \in \Z}|j| \log\log(|j|+e)F(\epsilon (j - c t))^2\right)^{1/2}.\end{aligned} 
 \end{equation*}
 The constant $C_\om$ here depends upon the realization and any estimate below will depend on the realization because of this step only.

Using the first inequality in Lemma \ref{calc lemma} with $a=j-ct$ and $b=ct$ and the triangle inequality
we get
\begin{equation*}\begin{aligned}\norm{ \chi(\cdot)F(\epsilon 
(\cdot - c t))}_{\ell^2} &\leq C_\om\left( 
\sum_{j \in 
\Z}|j-ct|\log\log(2|j-ct|+e)F(\epsilon 
(j- c t))^2 \right)^{1/2} \\ 
&+C_\om \sqrt{|t|
\log\log(2c|t|+e)}\| F(\epsilon( \cdot -  c 
t))\|_{\ell^2}.
\end{aligned} \end{equation*}
Call the two terms on the right $I$ and $II$. We estimate $II$ first.

Lemma 4.3  from \cite{Wright} shows that
$$
\|F(\ep(\cdot - ct))\|_{\ell^2} \le C \ep^{-1/2}\|F(\cdot - ct)\|_{H^1} =  C \ep^{-1/2}\|F\|_{H^1}
$$
and so
$$
II \le C_\om \ep^{-1/2} \sqrt{|t|
\log\log(2c|t|+e)}\|F\|_{H^1}.
$$
Then 
$$
\sup_{|t| \le T_0/\ep} II \le C_\om \ep^{-1} \sqrt{
\log\log(2cT_0\ep^{-1}+e)}\|F\|_{H^1}.
$$
Routine features of the logarithm show that $\log \log(2cT_0\ep^{-1}+e) \le C \log \log(1/\ep)$ 
 when $\ep \in (0,1/2)$ and so we have
$$
\sup_{|t| \le T_0/\ep} II \le C_\om \ep^{-1} \sqrt{
\log\log(1/\ep)}\|F\|_{H^1}.
$$

As for $I$, 
 using the second inequality  in Lemma \ref{calc lemma} with $|x| =2|j-ct|$ followed by the triangle inequality
 gets us:
\bes\begin{split}
 I\le &C_\om\left( 
\sum_{j \in 
\Z}|j-ct|\log(2\log(2\epsilon|j-ct|+e))F(\epsilon 
(j- c t))^2 \right)^{1/2}\\+
&C_\om\left( 
\sum_{j \in 
\Z}|j-ct|\log\log(\epsilon^{-1}+e) F(\epsilon 
(j- c t))^2 \right)^{1/2}.
\end{split}
\ees
Then we multiply by $\sqrt{\ep/\ep}$ and do some algebra to get:
\bes\begin{split}
I \le &C_\om \ep^{-1/2} \| \sqrt{\ep|\cdot-ct| \log(2 \log(2 |\ep(\cdot-ct)| +e))} F(\ep(\cdot-ct))\|_{\ell^2}\\
+ &C_\om \ep^{-1/2} \sqrt{\log \log(\ep^{-1} + e)} \| \sqrt{\ep|\cdot-ct| } F(\ep(\cdot-ct))\|_{\ell^2}.
\end{split}
\ees
Applying Lemma 4.3 from \cite{Wright} to 
tells us that
\bes\begin{split}
\| \sqrt{\ep|\cdot-ct| \log(2 \log(2 |\ep(\cdot-ct)| +e))} F(\ep(\cdot-ct))\|_{\ell^2}
\le & C \ep^{-1/2}\|F\|_{H^1_{LIL}}
\end{split}
\ees
and
$$
\| \sqrt{\ep|\cdot-ct| } F(\ep(\cdot-ct))\|_{\ell^2} \le 
 C\ep^{-1/2}\|F\|_{H^1_{sr}}
$$
and so we have,
$$
I \le C_\om \ep^{-1} \sqrt{\log \log(\ep^{-1} + e)} \|F\|_{H^1_{LIL}}\le C_\om \ep^{-1} \sqrt{\log \log(\ep^{-1})}\|F\|_{H^1_{LIL}}.
$$
Note that the right hand side does not depend on $t$ and so
$\sup_{|t| \le T_0/\ep} I \le C_\om \ep^{-1} \sqrt{\log \log(\ep^{-1})}$ and all together we have shown \eqref{big est 1}.

It happens that \eqref{big est 2} follows almost immediately from \eqref{big est 1} with some operator trickery.
For functions $G: \R \to \R$  and  $\ep \ne 0$ define the operator $\A_\ep$ via
$$
(\A_\ep G) (X) := {1 \over \ep} \int_X^{X+\ep} G(s) ds.
$$
We have
\be\label{average}
\| \A_\ep G\|_{H} \le C\|G\|_{H}
\ee
where $H$ may be $H^s$, $H^s_{LIL}$ or $H^s_{sr}$. Here comes the argument.
First we use Jensen's inequality to get:
\[
\| w(\cdot) \A_\ep G\|_{L^2}^2 = 
\int_{-\infty}^{\infty}w(X)^2 \left(\dfrac{1}{\epsilon}\int_X^{X+\epsilon}G(s)ds\right)^2dX \leq\int_{-\infty}^{\infty}
w(X)^2\dfrac{1}{\epsilon}\int_X^{X+\epsilon}G(s)^2dsdX.  \]
In the above $w(X)$ is a weight function.
If we change the order of integration we get
$$
\| w(\cdot) \A_\ep G\|_{L^2}^2
= \int_{-\infty}^\infty G(s)^2 {1 \over \ep} \int_{s-\ep}^s w(X)^2  dX ds.
$$
Let $\ds b_\ep(s) := {1 \over \ep w(s)^2} \int_{s-\ep}^s w(X)^2  dX$ so we have
$$
\| w(\cdot) \A_\ep G\|_{L^2}^2
= \int_{-\infty}^\infty w(s)^2 G(s)^2 b_\ep(s) dX ds \le \|b_\ep\|_{L^\infty} \| w(\cdot) G\|^2_{L^2}.
$$
If $w(X) = 1$, $w(X) = \sqrt{1 + |X|}$ or $w(X) = \sqrt{1+|X|\log \log (|X|+e)}$ it is easy to use the mean value theorem  
to show 
$\|b_\ep\|_{L^\infty} \le C$ when $\ep \in (0,1)$. With this, the last displayed inequality implies \eqref{average} (A little calculus shows that when $w(X)=\sqrt{1+|X|}$, then $\norm{b_\epsilon}_{L^{\infty}} \leq 3/2$ and thus in \eqref{average} $C \leq 3/2$.)


Continuing on in the proof of \eqref{big est 2}, the fundamental theorem of calculus tells us that $ F(X+\ep) - F(X)=\ep (\A_\ep F')(X)$. Thus:
$$
(\delta^+ F)(\ep (j-ct)) = F(\ep (j-ct) + \ep) - F(\ep(j-ct)) =\ep (\A_\ep F') (\ep(j-ct)).
$$
In which case we see that
$$
\| \chi(\cdot) \delta^+F(\ep(\cdot -ct))\|_{\ell^2} = 
\ep \| \chi(\cdot) (\A_\ep F')(\ep(\cdot-ct))\|_{\ell^2}.
$$
We have produced an extra factor of $\ep$!
Using \eqref{big est 1} and \eqref{average}
$$
\| \chi(\cdot) \delta^+F(\ep(\cdot -ct))\|_{\ell^2} \le {C}_\om \sqrt{\log \log(\ep^{-1})}\|\A_\ep F'\|_{H^1_{LIL}}
\le {C}_\om \sqrt{\log \log(\ep^{-1})}\|F\|_{H^2_{LIL}}.
$$
That is \eqref{big est 2} and does it for this proof.


\end{proof}

Now we can prove:
\begin{proposition} \label{residuals}
Fix $A,B \in H^3_{LIL}$ and take $\tilde{r}_\ep$ and $\tilde{p}_\ep$ as in \eqref{ansatz2}.
Fix  $T_0>0$.
Then for 
almost every realization of $\{k(j)\}$ and $\{m(j)\}$ there is a finite positive constant
$
C_\om = C_\om\left(k,m,T_0,\|A\|_{H^3_{LIL}},\|B\|_{H^3_{LIL}}\right)
$
for which $\ep \in (0,1/2)$ implies
\be\label{res ests}
\sup_{|t|\le T_0/\ep}\left( \|\Res_1(\tilde{r}_\ep,\tilde{p}_\ep)\|_{\ell^2} + \|\Res_1(\tilde{r}_\ep,\tilde{p}_\ep)\|_{\ell^2} \right)
\le C_\om \ep \sqrt{\log \log(1/\ep)}.
\ee
\end{proposition}

\begin{proof} 
We prove the estimate for the piece involving $\Res_1$ as
 the other part is all but identical. A tedious calculation shows that 
 \begin{equation}\label{res1 gory}
  \begin{aligned} \text{Res}_1(\tilde{r}_\ep,\tilde{p}_\ep)&= 
\frac{1}{\sqrt{\tilde{k}\bar{m}}}\left( -\delta^+[A(\epsilon (j-ct))]+ \ep A'(\epsilon (j-ct))\right)\\
&+\frac{1}{\sqrt{\tilde{k}\bar{m}}}\left( -\delta^+[B(\epsilon (j+ct))]+ \ep B'(\epsilon (j+ct))\right)\\
&
+\frac{c\epsilon^2\chi_m(j)}{k(j)}A'(\epsilon(j-ct)) 
+\frac{c\epsilon^2\chi_m(j)}{k(j)}B'(\epsilon(j+ct))\\
&- \frac{\epsilon 
\chi_k(j+1)}{\sqrt{\tilde{k}\bar{m}}}\delta^+[A'(\epsilon (j -ct))]
- \frac{\epsilon 
\chi_k(j+1)}{\sqrt{\tilde{k}\bar{m}}}\delta^+[B'(\epsilon (j +ct))].
\end{aligned} \end{equation}
The terms in the first two lines are fully deterministic and estimable using Lemma 4.3 of \cite{Wright}.
Specifically the $\ell^2$ norm of 
each is controlled by
$$
C \ep^{3/2} \left(\|A\|_{H^2} + \|B\|_{H^2} \right) 
$$
for $|t| \le T_0/\ep$.
This is dominated by the right hand side of \eqref{res ests}.
Using \eqref{big est 1} we see that the $\ell^2$ norm in the third line is controlled by
$$
\ep^2 \left(C_\om \ep^{-1} \sqrt{\log \log(1/\ep)} \left(\|A'\|_{H^1_{LIL}} +\|B'\|_{H^1_{LIL}}\right)\right)
$$
for $|t| \le T_0/\ep$.
Again this is dominated by the right hand side of \eqref{res ests}.
Similarly we use \eqref{big est 2} to handle the terms in the last line, which are controlled by
$$
\ep \left(C_\om  \sqrt{\log \log(1/\ep)} \left(\|A'\|_{H^2_{LIL}} +\|B'\|_{H^2_{LIL}}\right)\right).
$$
It is here we see why $H^3_{LIL}$ is needed in \eqref{res ests}. 
\end{proof}

\begin{remark} \label{compare remark}
We quickly note that if the $k(j)$ and $m(j)$ are constant with respect to $j$, one can easily
chase through this proof and see that the estimate size of the residuals decreases to $C \ep^{3/2}$. Likewise
if the springs and masses vary periodically, one finds that $\chi_m(j)$ and $\chi_k(j)$ are in $\ell^\infty$
and then this proof would demonstrate the size of the residuals is bounded by $C \ep^{1/2}$.
\end{remark}

\subsection{Boundedness in Mean}

The almost sure boundedness does not provide us with any kind of description for the $\omega$ dependent constant $C_\omega$. In this section we estimate the error in mean, finding estimates in terms of $\sigma_m$ and $\sigma_k$.
\begin{lemma} 
Let $y(j)$ and $\chi(j)$ be as in Theorem \ref{LIL} and $n \in \Z^+\cup\{0\}$. Consider the process
\begin{equation*}W_j(n):= \chi(j+n)-\chi(j).\end{equation*} Then, for every $j$,  $W_j(n)$ is a martingale
in the variable $n$ and, for any $N >0$,
\begin{equation}
    \E[\max_{0 \leq n \leq N}(W_j(n))^{2}] \leq 4N\sigma^{2}.
\end{equation}

\end{lemma}

\begin{proof}
From the definition of $\chi$, we have
\begin{equation}\label{W}
W_j(n)=y(j)+\cdots+y(j+n-1).\end{equation} Then \begin{equation*}
  \E[|W_j(n)|] \leq C |n|.  
\end{equation*}

Conditioning upon $W_j(n)$ gives
\begin{align*}
\E[W_j(n+1)|W_j(n)]=\E[y(j+n)+W_j(n)|W_j(n)]=W_j(n).   \end{align*}
This proves that $W_j$ is a martingale. From the basic theory of martingales this tells us $|W_j|^2$ is a submartingale. It follows from the $L^{p}$ maximum inequality, see \cite{Durret}, and a direct computation using \eqref{W}
\begin{equation*} 
\E\left[\max_{0 \leq n \leq N}(W_j(n))^2 \right]\leq 4\E[|W_j(N)|^2]=4N\sigma^2.
 \end{equation*}
 \end{proof}
 \begin{remark}
 \label{Rmk WLOG}
 We can define a similar process $W_j(n):=\chi(j-n)-\chi(j)$ which would have exactly the same properties but with a different version of \eqref{W} i.e.
 \[W_j(n)=y(j-n)+\cdots+y(j-1).\] This symmetry allows us to handle positive and negative times with the same argument.
 \end{remark}
 We use the following corollary in the results that follow.
 \begin{corollary}
 \label{martingale}
 \begin{equation*}
     \chi_k(j+n)-\chi_k(j) \mand \chi_m(j+n)-\chi_m(j)
 \end{equation*}
 are martingales in $n$ with 
 \begin{equation*}
  \E\left[\max_{0 \leq n \leq N} \left(\chi_k(j+n)-\chi_k(j)\right)^2\right]  \leq 4N \tilde{k}\sigma_k^2 \mand \  \E\left[\max_{0 \leq n \leq N} \left(\chi_k(j+n)-\chi_k(j)\right)^2 \right] \leq 4N \dfrac{\sigma^2_m}{\bar{m}}.
 \end{equation*} 
 
 \end{corollary}
 
 We have now gotten the necessary probability out of the way to prove the following lemma, analogous to Lemma \ref{Big Lemma}, but in expectation.

\begin{lemma}
For any $T_0 > 0$ and $\epsilon \in (0,1/2)$ the following inequalities hold
\begin{equation}
\label{meanboundlemma1}
\E\left[\sup_{|t| \leq T_0/\epsilon}\norm{\chi(\cdot)F(\epsilon(\cdot-ct))}_{\ell^2}\right] \leq 2\sqrt{2}\epsilon^{-1}\sigma\max\{2\sqrt{|c| T_0},1\}\norm{F}_{H^2_{sr}} 
 \end{equation}
 and 
 \begin{equation}
 \label{meanboundlemma2}
     \E\left[\sup_{|t| \leq T_0/\epsilon}\norm{\chi(\cdot)\delta^\pm F(\epsilon(\cdot-ct))}_{\ell^2}\right] \leq 3\sqrt{2}\sigma\max\{2\sqrt{|c| T_0},1\}\norm{F}_{H^3_{sr}}. 
 \end{equation}
 In the above $\chi$ is either $\chi_k$ or $\chi_m$ with $\sigma$ either $\sigma_k\sqrt{\tilde{k}}$ or $\dfrac{\sigma_m}{\sqrt{\bar{m}}}$ respectively.
\end{lemma}

\begin{proof}
Without loss of generality (see Remark \ref{Rmk WLOG}) let $t \in \R^+\cup\{0\}$. Write $ct= \floor{ct}+ \alpha$ where $\alpha \in [0,1).$ Let $n \in \Z$ in the following. We start with the inequality
\begin{equation*}\sup_{0 \leq t \leq T_0/\epsilon}\norm{\chi(\cdot)F(\epsilon \cdot - \epsilon c t)}^2_{\ell^2}\leq \sup_{0 \leq n \leq cT_0/\epsilon, \alpha \in [0,1)}\sum_{j \in \Z}\chi(j)^2F(\epsilon j-\epsilon  n-\epsilon \alpha)^2. \end{equation*}

The inequality is due to the fact that for any $t \in [0,\lfloor T_0/\epsilon\rfloor+1)$ there exists an $n \in [0, cT_0/\epsilon]$ and $\alpha \in [0,1)$ s.t. $n+\alpha =ct$, which is a slightly greater range for $t$ than we initially cared about. Using the Mean Value Theorem, we have that
\begin{equation*}
    F(\epsilon j-\epsilon n-\epsilon \alpha)=F(\epsilon j-\epsilon n)-\epsilon \alpha F'(x_j)
\end{equation*}
where $x_j \in (\epsilon j-\epsilon n -\epsilon \alpha,\epsilon j-\epsilon n).$ Substituting this in and using the basic inequality $(a+b)^2 \leq 2(a^2+b^2)$, we get
\begin{equation*}\begin{aligned}\sup_{0 \leq t \leq T_0/\epsilon}\norm{\chi(\cdot)F(\epsilon \cdot - \epsilon c t)}^2_{\ell^2} &\leq \sup_{0 \leq n \leq cT_0/\epsilon, \alpha \in [0,1)}\sum_{j \in \Z}\chi(j)^2(F(\epsilon j-\epsilon  n)-\epsilon \alpha F'(x_j) )^2. 
 \\&\leq \sup_{0 \leq n \leq cT_0/\epsilon, \alpha \in [0,1)}2\sum_{j \in \Z}\chi(j)^2F(\epsilon j-\epsilon n)^2+\chi(j)^2(\epsilon \alpha F'(x_j) )^2.  \end{aligned} \end{equation*}
Call the expectation of the first term $I$ and the second's expectation $II$. Note that $I$ does not depend upon $\alpha$. We find from a change of indices, that
\begin{equation}\begin{aligned}I&=\E\left[\sup_{0 \leq n \leq cT_0/\epsilon}2\sum_{j \in \Z}\chi(j+n)^2F(\epsilon j)^2 \right]\\ &=\E\left[\sup_{0 \leq n \leq cT_0/\epsilon}2\sum_{j \in \Z}(\chi(j+n)-\chi(j)+\chi(j))^2F(\epsilon j)^2 \right]\end{aligned} \end{equation}
Using the same basic inequality as above we get
\begin{equation}
\label{firsttermstart}
I \leq \E\left[ \sup_{0 \leq n \leq cT_0/\epsilon} 4\sum_{j \in \Z}\left(\chi(j)^2+(\chi(j+n)-\chi(j))^2\right)F(\epsilon j)^2\right]. 
\end{equation} The supremum sees only the term with $n$, and Fubini's theorem allows the expected value to pass through the sum. And so
\begin{equation*}
    I\leq 4\sum_{j \in \Z} \left(\E\left[\chi(j)^{2}\right]+\E\left[\sup_{0\leq n\leq cT_0/\epsilon}(\chi(j+n)-\chi(j))^2\right]\right)F(\epsilon j)^2.
\end{equation*}
A direct computation on the first term using the definition of $\chi$ and using Corollary \ref{martingale} on the second term we find 
\begin{equation}
\label{firsttermend}
    I \leq 4\sum_{j \in \Z}\left(\sigma^2|j|+4\sigma^2cT_0\epsilon^{-1}\right)F(\epsilon j)^2.
\end{equation}
According to Lemma 4.3 and 4.4 from \cite{Wright}, $I$ is dominated by 
\begin{equation}
    \label{dominate}
    8\epsilon^{-2}\sigma^2 \max\{4cT_0,1\}\norm{F}^2_{H^1_{sr}}.
\end{equation}
Now we turn our attention to $II$. We can eliminate the $\alpha$ dependence by taking $\alpha=1$ i.e. choose $\tilde{x}_j$ s.t. \begin{equation*}F'(\tilde{x}_j)=\max_{x \in[\epsilon j- \epsilon n-\epsilon, \epsilon j-\epsilon n]}F'(x). \end{equation*}
Then 
\begin{equation*}(\epsilon \alpha F'(x_j))^2 \leq (\epsilon F'(\tilde{x}_j))^2. \end{equation*}
Shifting the index by $n$ we get
\begin{equation*} II=\E\left[\sup_{0 \leq n \leq cT_0/\epsilon}\sum_{j \in \Z}\chi(j+n)^2F'(\tilde{x}_{j+n})^2\right] \end{equation*} where $\tilde{x}_{j+n} \in [\epsilon j -\epsilon, \epsilon j]$ does not depend on $n$. We therefore may relabel $\tilde{x}_j=\tilde{x}_{j+n}.$ We use the same steps here as we used from to \eqref{firsttermstart} to \eqref{firsttermend}.
\begin{equation*}II \leq 4\sum_{j \in \Z}(\sigma^2|j|+4\sigma^2cT_0\epsilon^{-1})\epsilon^2F'(\epsilon \tilde{x}_j)^2.
\end{equation*}
Again, by Lemma 4.3 from \cite{Wright}, $II$ is  dominated by \begin{equation}
\label{dominated2}
   8\sigma^2 \max\{4cT_0,1\}\norm{F}^2_{H^2_{sr}}.
\end{equation}
By \eqref{dominate} and \eqref{dominated2} we have

\begin{equation*}
    \E\left[\sup_{0 \leq t\leq T_0/\epsilon}\norm{\chi(\cdot)F(\epsilon\cdot-\epsilon ct}^2\right] \leq  8\epsilon^{-2}\sigma^2 \max\{4cT_0,1\}\norm{F}^2_{H^2_{sr}}.
\end{equation*}

An standard application of Jensen's inequality yields
\begin{equation*}
\E\left[\sup_{0\leq t\leq T_0/\epsilon}\norm{\chi(\cdot)F(\epsilon\cdot-\epsilon ct}\right]  \leq 2\sqrt{2}\epsilon^{-1}\sigma\max\{2\sqrt{c T_0},1\}\norm{F}_{H^2_{sr}}.
\end{equation*}
This proves \eqref{meanboundlemma1}.

The exact same trickery that was used in Lemma \ref{Big Lemma}
works to prove \eqref{meanboundlemma2}. Using \eqref{meanboundlemma1} and then \eqref{average}
\begin{equation*}
\E\left[\sup_{0\leq t\leq T_0/\epsilon} \norm{\chi(\cdot)\delta^+F(\epsilon(\cdot-ct))}_{\ell^2}\right] \leq 3\sqrt{2}\sigma \max\{2\sqrt{c T_0},1\}\norm{F}_{H^3_{sr}}.
\end{equation*}
 This shows \eqref{meanboundlemma2}.
\end{proof}

\begin{remark}
The functions in this subsection are required to be once more differentiable than the functions in the previous subsection, due to the use of the Mean Value Theorem in the beginning of the proof of the previous lemma. 
\end{remark}
Now we can prove:

\begin{proposition}
\label{Res in Mean}
Fix $A,B \in H^{4}_{sr}$ and take $\tilde{r}_\epsilon$ and $\tilde{p}_\epsilon$ as in \eqref{ansatz2}. Fix $T_0>0$. For there exists a positive constant $C(\tilde{k},\bar{m},a_k,b_k,a_m,b_m,T_0,\norm{A}_{H_{sr}^4},\norm{B}_{H_{sr}^4})$ 
for which $\ep \in (0,1/2)$ implies
\begin{equation}
\label{mean bound}
   \E\left[ \sup_{|t|\le T_0/\ep}\left( \|\Res_1(\tilde{r}_\ep,\tilde{p}_\ep)\|_{\ell^2} + \|\Res_2(\tilde{r}_\ep,\tilde{p}_\ep)\|_{\ell^2} \right)\right]
\le C\epsilon\left(\epsilon^{1/2}+\max\{\sigma_m, \sigma_k\}\right).
\end{equation}

\end{proposition}

\begin{proof}
The proof begins the same way as the proof for Proposition \ref{residuals} except now we take expectation of \eqref{res1 gory}. 
Since the first two lines of \eqref{res1 gory} are deterministic,  using Lemma 4.3 in \cite{Wright}, they are controlled by 
\[
 \epsilon^{3/2}\dfrac{\sqrt{2}}{\sqrt{\tilde{k}\bar{m}} }\left(\|A\|_{H^2} + \|B\|_{H^2} \right).\]
Next use \eqref{meanboundlemma1} to control the third line with
\[\dfrac{2\sqrt{2}\epsilon \sigma_mc\max\{2\sqrt{|c|T_0},1\}}{a_k\sqrt{\bar{m}}}\left(\norm{A'}_{H_{sr}^2}+\norm{B'}_{H_{sr}^2}\right),\] which is dominated by \eqref{mean bound}. We use \eqref{meanboundlemma2}
to estimate the fourth line:
\[\dfrac{3\sqrt{2}\epsilon\sigma_k \max\{2\sqrt{|c|T_0,}1\}}{\sqrt{\bar{m}}}\left(\norm{A'}_{H_{sr}^3}+\norm{B'}_{H_{sr}^3}\right).\]
As before, the estimate for $\Res_2$ follows a parallel argument and is omitted.
\end{proof}

 
 

 
 \section{Error estimates}
 \label{Err est sec}
 In this section we prove rigorous estimates 
 using ``energy'' arguments, similar to \cite{Wright, Chirilius-Bruckner, Schneider-Wayne}.
 
 \subsection{The energy argument}
 
Let $r$ and $p$ be a true solution to \eqref{FORFPUT} and take $\tilde{r}_\ep$
and $\tilde{p}_\ep$ as in \eqref{ansatz2}. Define {\it error functions} $\eta$ and $\xi$ implicitly by
\begin{equation}\label{erf}
    r =\tilde{r}_\epsilon + \frac{\eta}{k} \mand
    p=\tilde{p}_\epsilon+\xi.
\end{equation}
It is our goal to determine the size in $\ell^2$ of $\eta$ and $\xi$ 
during the period $|t| \le T_0/\ep$.
To that end, insert \eqref{erf} into \eqref{FORFPUT} to find that 
\begin{equation}\label{erreq}
\begin{aligned}  
\frac{\dot{\eta}}{k}&=\delta^+\xi + \Res_1\\
m\dot{\xi}&=\delta^-\eta + \Res_2
\end{aligned} \end{equation} 
where 
$\Res_1 = \Res_1(\tilde{r}_\ep,\tilde{p}_\ep)$
and 
$
\Res_2 = \Res_2(\tilde{r}_\ep,\tilde{p}_\ep)
$
as in \eqref{res def}.

Next 
define the {\it energy} to be
\begin{equation*}H(t) \coloneqq\frac{1}{2}\sum_{j \in \Z} \left[k(j)^{-1}\eta^2(j,t)+m(j)\xi^2(j,t)\right].\end{equation*} 
Since we have assumed that the $k(j)$ and $m(j)$ are drawn from distributions with support in
$[a_k,b_k] \subset \R^+$ and $[a_m,b_m] \subset \R^+$, respectively, a short calculation shows that
$\sqrt{H}$ is equivalent to $\|\eta,\xi\|_{\ell^2 \times \ell^2}$ and the constants of equivalence depend only
on $a_k,a_m,b_k$ and $b_m$. That is to say, the equivalence is realization independent.

Time differentiation of $H$ gives 
\begin{equation*}\dot{H} = \sum_{j\in \Z}\left[k^{-1}\eta\dot{\eta}+m\xi\dot{\xi}\right]. \end{equation*}
Using \eqref{erreq}
\begin{equation*}\dot{H}=\sum_{j \in \Z}\left[\eta(\delta^+\xi + \text{Res}_1)+\xi(\delta^-\eta + \text{Res}_2)\right] .\end{equation*}
Summing by parts:
\begin{equation*} \dot{H} =\sum_{j \in \Z}\left[\eta\text{Res}_1+\xi\text{Res}_2\right].\end{equation*} 
Cauchy-Schwarz implies that 
\begin{equation*}\dot{H} \leq \norm{\text{Res}_1,\text{Res}_2}_{\ell^2 \times \ell^2}\norm{\eta,\xi}_{\ell^2 \times \ell^2}. \end{equation*} 
Then we use the equivalence of $\sqrt{H}$ and $\norm{\eta,\xi}_{\ell^2 \times \ell^2}$ to get:
\begin{equation*}\dot{H} \leq C\norm{\text{Res}_1,\text{Res}_2}_{\ell^2 \times \ell^2}\sqrt{H}. \end{equation*} 

Set 
\begin{equation*}\Gamma_\epsilon \coloneqq \sup_{|t|\leq T_0/\epsilon}\norm{\text{Res}_1,\text{Res}_2}_{\ell^2 \times \ell^2},  \end{equation*} so $\dot{H}/\sqrt{H} \leq C\Gamma_\epsilon$. We integrate from $0$ to $ t$
\begin{equation*} 2\sqrt{H(t)} \leq 2\sqrt{H(0)} +C\Gamma_\epsilon t.\end{equation*}
And so, for $t \leq T_0/\epsilon$, we have
\begin{equation*}
\sqrt{H(t)} \leq \sqrt{H(0)}+C\Gamma_\epsilon T_0\epsilon^{-1}.  \end{equation*}
If we use the equivalence of the $\sqrt{H}$ and $\|\eta,\xi\|_{\ell^2 \times \ell^2}$ once again, we find that we have proven
\begin{equation}\label{finish line}
\sup_{|t| \le T_0/\ep} \|\eta(t),\xi(t)\|_{\ell^2 \times \ell^2} \leq  C\|\eta(0),\xi(0)\|_{\ell^2 \times \ell^2}+C\Gamma_\epsilon\epsilon^{-1}.
\end{equation}
A key feature of the above inequality is that the only place where the specific realization of the springs and masses enters is through $\Gamma_\ep$.


\subsection{Almost sure error estimates}
We can now prove our first main theorem, which is about almost sure estimation of the absolute error:
\begin{theorem}\label{absolute theorem}
Fix $\Phi,\Psi \in H^3_{LIL}$ and $T_0>0$. Let $r$ and $p$ be the solution of \eqref{FORFPUT}
with initial data
$$
r(j,0) = \Phi(\ep j)/k(j) \mand p(j,0) = \Psi(\ep j).
$$
For almost every realization of $\{k(j)\}$ and $\{m(j)\}$ there is a finite positive constant
$$
C_\om = C_\om(k,m,a_k,b_k,a_m,b_m,\|\Phi\|_{H^3_{LIL}},\|\Psi\|_{H^3_{LIL}})
$$
for which $\ep \in (0,1/2)$ implies
$$
\sup_{|t| \le T_0/\ep} \left \| r(\cdot, t) - {1 \over k(\cdot)}  \left( A(\ep (\cdot-ct) + B(\ep(\cdot + ct)\right) \right \|_{\ell^2} \le 
C_\om \sqrt{\log \log(1/\ep)}
$$
and
$$
\sup_{|t| \le T_0/\ep}\left \| p(\cdot, t) -  {1 \over \sqrt{\tilde{k} \bar{m}}} \left( -A(\ep (\cdot-ct) +B(\ep(\cdot + ct)\right) \right \|_{\ell^2} \le 
C_\om \sqrt{\log \log(1/\ep)}.
$$
In the above
$$
A(X) :=  {1 \over 2} \Phi(X) - {\sqrt{\tilde{k} \bar{m}} \over 2} \Psi(X)
\mand 
B(X) := {1 \over 2} \Phi(X) + {\sqrt{\tilde{k} \bar{m}} \over 2} \Psi(X).
$$

\end{theorem}
\begin{remark} In the case where the masses and springs vary periodically instead of randomly, the size of the error decreases to $C \ep^{1/2}$; in fact the proof we supply in a moment together with Remark \ref{compare remark} suffices to demonstrate this. Likewise, if the masses and springs are constant a slightly modified version of 
the proof can be used to decrease the error to $C \ep^{3/2}$. It is this extra wiggle room in the error in these cases which opens the door to longer time scales and KdV-like approximations.
\end{remark}

\begin{proof}
Form $\tilde{r}_\ep$ and $\tilde{p}_\ep$ from the functions $A$ and $B$ as specified in \eqref{ansatz2} and
$\eta$ and $\xi$ as in \eqref{erf}. A bit of algebra shows that
\be\label{ee ic}
\eta(j,0) = \ep{\chi_m(j) } \left(A'(\ep j) + B' (\ep j) \right)
\mand
\xi(j,0) = \ep{\chi_k(j) \over \sqrt{\tilde{k}\bar{m}}} \left(-A'(\ep j ) + B' (\ep j) \right).
\ee
Using \eqref{big est 1} in a very crude way, we see that almost surely
$$
\| \eta(0),\xi(0)\|_{\ell^2 \times \ell^2} \le C_\om \sqrt{\log\log(1/\ep)}
$$
with the constant depending on $\|A\|_{H^2_{LIL}}$ and $\|B\|_{H^2_{LIL}}$.
We estimated $\Gamma_\ep$ in Proposition \ref{residuals} and found that 
$\Gamma_\ep \le C_\om \ep \sqrt{\log\log(1/\ep)}$ when $\ep \in (0,1/2)$ almost surely.
Therefore \eqref{finish line} gives
$$
\sup_{|t|\le T_0/\ep} \|\eta(t) ,\xi(t)\|_{\ell^2 \times \ell^2} \leq C_\om  \sqrt{\log\log(1/\ep)}.
$$
To finish the proof we note that the triangle inequality tells us
\bes\begin{split}
&\left \| r(\cdot, t) - {1 \over k(\cdot)}  \left( A(\ep (\cdot-ct)) + B(\ep(\cdot + ct))\right) \right \|_{\ell^2}\\
 \le &
\left \| r( t) - \tilde{r}_\ep(t)\right\|_{\ell^2} +\left \| \tilde{r}_\ep(\cdot,t)-{1 \over k(\cdot)}  \left( A(\ep (\cdot-ct)) + B(\ep(\cdot + ct))\right) \right \|_{\ell^2}\\
\le &C\|\eta(t)\|_{\ell^2} + C\ep  \| {\chi_m(\cdot) } A'(\ep (\cdot-ct))\| +  C\ep \| {\chi_m(\cdot) } B' (\ep (\cdot -ct)) \|_{\ell^2}.
\end{split}\ees
The terms involve $A$ and $B$ can be estimated using \eqref{big est 1} by $C_\om \sqrt{\log \log(1/\ep)}$
so we find
$$
\sup_{|t|\le T_0/\ep} \left \| r(\cdot, t) - {1 \over k(\cdot)}  \left( A(\ep (\cdot-ct)) + B(\ep(\cdot + ct))\right) \right \|_{\ell^2}
\le C_\om \sqrt{\log \log(1/\ep)}.
$$
The remaining estimate in the Theorem \ref{absolute theorem} is shown by a parallel argument and is omitted.
\end{proof}

It may seem like the estimates in Theorem \ref{absolute theorem} are utterly useless since the size of the error diverges
as $\ep \to 0^+$. But the error in that theorem is the {\it absolute} error; the {\it relative} error does in fact 
vanish in the limit.
\begin{corollary} \label{relative cor}
Under the same conditions as in Theorem \ref{absolute theorem} we almost surely have
\begin{equation*}
\lim_{\ep \to 0^+}
\sup_{|t| \le T_0/\ep} {\left \| r(\cdot, t) - {1 \over k(\cdot)}  \left( A(\ep (\cdot-ct)) + B(\ep(\cdot + ct))\right) \right \|_{\ell^2} \over \|r(t)\|_{\ell^2} } = 0
\end{equation*}
and
\begin{equation*}
\lim_{\ep \to 0^+}
\sup_{|t| \le T_0/\ep}
{\left \| p(\cdot, t) -  {1 \over \sqrt{\tilde{k} \bar{m}}} \left( -A(\ep (\cdot-ct)) +B(\ep(\cdot + ct))\right) \right \|_{\ell^2} 
\over \|p(t)\|_{\ell^2}} = 0.
\end{equation*}
\end{corollary}
\begin{proof}
The reverse triangle inequality gives
$$
\|r(t)\|_{\ell^2} \ge \left \|{1 \over k(\cdot)}  \left( A(\ep (\cdot-ct)) + B(\ep(\cdot + ct))\right) \right \|_{\ell^2} -
\left \| r(\cdot, t) - {1 \over k(\cdot)}  \left( A(\ep (\cdot-ct) + B(\ep(\cdot + ct)\right) \right \|_{\ell^2} .
$$
Using Lemma 4.3 from \cite{Wright} for the first term and Theorem \ref{absolute theorem} for the second
we obtain
$$
\|r(t)\|_{\ell^2} \ge C \ep^{-1/2} -C_\om \sqrt{\log \log(1/\ep)}
$$
for all $|t| \le T_0/\ep$.
This is positive for $\ep$ small enough and so we get the first limit in the corollary by dividing
the absolute error for $r$ in Theorem \ref{absolute theorem} by this estimate and taking the limit. The second limit is analogous.
\end{proof}
\subsection{Error estimate in mean} 
We can now prove our second main theorem, which is an estimate of the mean of the error.

\begin{theorem}
\label{main mean theorem}
Fix $\Phi, \Psi \in H^{4}_{sr}$ and $T_0>0.$ Let $r$ and $p$ be the solution of \eqref{FORFPUT} with initial data
\[r(j,0)=\Phi(\epsilon j)/k(j) \mand p(j,0)=\phi(\epsilon j). \]
There exists a positive constant $C(\tilde{k},\bar{m},a_k,b_k,a_m,b_m,T_0,\norm{A}_{H^{4}_{sr}},\norm{B}_{H^4_{sr}})$ for which $\ep \in (0,1/2)$ implies
\[\E\left[\sup_{|t| \le T_0/\ep} \left \| r(\cdot, t) - {1 \over k(\cdot)}  \left( A(\ep (\cdot-ct) + B(\ep(\cdot + ct)\right) \right \|_{\ell^2}  \right] \leq C\left(\epsilon^{1/2}+\max\{\sigma_m,\sigma_k\}\right)\]
and 
\[\E\left[\sup_{|t| \le T_0/\ep}\left \| p(\cdot, t) -  {1 \over \sqrt{\tilde{k} \bar{m}}} \left( -A(\ep (\cdot-ct) +B(\ep(\cdot + ct)\right) \right \|_{\ell^2} \right] \leq C\left(\epsilon^{1/2}+\max\{\sigma_m,\sigma_k\}\right).\]

In the above
\[A(X) :=  {1 \over 2} \Phi(X) - {\sqrt{\tilde{k} \bar{m}} \over 2} \Psi(X)
\mand 
B(X) := {1 \over 2} \Phi(X) + {\sqrt{\tilde{k} \bar{m}} \over 2} \Psi(X). \]
\end{theorem}
\begin{proof}
Begin as in the proof of Theorem \ref{absolute theorem}.
Using \eqref{meanboundlemma1} on \eqref{ee ic}
\[\norm{\eta(0),\xi(0)}_{\ell^2\times\ell^2} \leq C\max\{ \sigma_m,\sigma_k \} \] with constant $C$ depending on $\norm{A}_{H^3_{sr}}$ and $\norm{B}_{H^3_{sr}}.$
Proposition \ref{Res in Mean} gives us 
\[\E\left[\Gamma_\ep \right] \le C\epsilon\left(\epsilon^{1/2}+  \max\{\sigma_m, \sigma_k\}\right)\] when $\ep \in (0,1/2)$.
Therefore \eqref{finish line} gives
$$
\E\left[\sup_{|t|\le T_0/\ep} \|\eta(t) ,\xi(t)\|_{\ell^2 \times \ell^2}\right] \le C\left(\epsilon^{1/2}+  \max\{\sigma_m, \sigma_k\}\right).
$$
To finish the proof we note that the triangle inequality tells us
\bes\begin{split}
 &\E\left[\sup_{|t|\leq T_0/\epsilon} \left\| 
r(\cdot, t) - {1 \over k(\cdot)}  \left( A(\ep 
(\cdot-ct)) + B(\ep(\cdot + ct))\right)  \right\|_{\ell^2}\right]\\
 \le & \E \left[ \sup_{|t|\leq T_0}
 \left \| r( t) - \tilde{r}_\ep(t)\right\|_{\ell^2}\right] +\E\left[ \sup_{|t| \leq T_0}\left \| \tilde{r}_\ep(\cdot,t)-{1 \over k(\cdot)}  \left( A(\ep (\cdot-ct)) + B(\ep(\cdot + ct))\right) \right \|_{\ell^2}\right]\\
 \le & \E\left[\sup_{|t| \leq T_0}C\|\eta(t)\|_{\ell^2} \right] + \E \left[\sup_{|t| \leq T_0} C\ep  \| {\chi_m(\cdot) } A'(\ep (\cdot-ct))\| \right]+ \E\left[\sup_{|t| \leq T_0} C\ep \| {\chi_m(\cdot) } B' (\ep (\cdot -ct)) \|_{\ell^2}\right].
\end{split}\ees
The $C$ depends on $a_k,b_k,a_m$, and $b_m$, which are fixed, so we may pull it out of the expected value. The terms that involve $A$ and $B$ can be estimated using \eqref{meanboundlemma1} by $C\max\{\sigma_k,\sigma_m\} $
so we find
$$
\E\left[ \sup_{|t|\le T_0/\ep} \left \| r(\cdot, t) - {1 \over k(\cdot)}  \left( A(\ep (\cdot-ct)) + B(\ep(\cdot + ct))\right) \right \|_{\ell^2}\right]
\le C\left(\epsilon^{1/2}+\max\{\sigma_k,\sigma_m\}\right).
$$
The remaining estimate in the Theorem \ref{main mean theorem} is shown by a parallel argument and is omitted.

\end{proof}

\section{Coarse-graining}
\label{Coarse grain sec}
We now prove strong convergence results using the 
 ideas of coarse-graining from \cite{Mielke}.
 We need quite a few tools.
Letting $f: \Z \to \R$ and  $g,u,v: \R \to \R$ define
%
\begin{equation*}\begin{aligned}
    F[f](\kappa) &: = \frac{1}{2\pi} \sum_{j \in \Z} e^{-ij\kappa}f(j) 
    \\
    F^{-1}[g](j)&:=\int_{-\pi}^{\pi}g(\kappa)e^{i\kappa j}
    \\
    \mathcal{F}[u](\xi)&:=\frac{1}{2\pi}\int_{\R}u(x)e^{-i\xi x}dx
     \\
     \mathcal{F}^{-1}[v](x)&: = \int_{\R} v(\xi)e^{i\xi x}d\xi \\
    \theta_\phi(\kappa) &:=\begin{cases} 1 & \kappa \in (-\phi ,\phi)\\
    0 & \text{else}
    \end{cases}
    \\
    \mathcal{L}[f](x)&:= \mathcal{F}^{-1}[\theta_\pi (\cdot) F[f](\cdot)](x) \\
     \mathcal{S}[u](j)&:=u(j).
    \end{aligned}\end{equation*}
These are, in order, the Fourier Transform for sequences, its inverse, the Fourier transform of functions $\R \to \R$, its inverse, the indicator function of $(-\phi,\phi)$, a ``low pass'' interpolation operator, and a sampling operator.    
To be clear, in the above $x, \xi, \kappa \in \R$ and $j \in \Z$, always.
    
The operator $\mathcal{L}$ converts a sequence $f$ defined on $\Z$ to a new function defined on $\R.$ The sampling function $\mathcal{S}$ returns a sequence from a function defined on $\R$. It is an easy exercise to show that $\mathcal{S L}[f](j)=f(j)$ so it is clear $\mathcal{L}$ is an interpolation operator. Another essential property is the following. 

\begin{lemma}\label{L is good}
Let $f$ be a sequence in $\ell^2$. Then
\begin{equation*}\norm{f}_{\ell^2} =2 \pi \norm{\mathcal{L}[f]}_{L^2(\R)}.\end{equation*}
\end{lemma}

\begin{proof} By Plancherel's theorem:
\begin{equation*}\norm{F[f]}_{L^2(-\pi,\pi)}=\frac{1}{2\pi} \norm{f}_{\ell^2}. \end{equation*}
Then by a slightly different Plancherel's theorem:
\begin{equation*}\norm{\mathcal{F}^{-1}[\theta_\pi F[f]]}_{L^2(\R)} = \norm{\theta_\pi F[f]}_{L^2 (\R)}= \norm{ F[f]}_{L^2 (-\pi,\pi)}\end{equation*} completing the proof. \end{proof}

We need one more lemma before we can state the strong convergence results. 
It states that for a smooth enough function, the more frequently it is sampled,
the more is interpolation looks like the original function.
\begin{lemma}\label{interp}
Let $f:\R \to \R$ be in  $H^s$ with $s>1/2$ and put $f_\ep(x) = f(\ep x)$.
Then 
\begin{equation*}\lim_{\epsilon \to 0^+}\norm{\mathcal{L}\mathcal{S}[f_\ep](\cdot/\ep) -f}_{L^2}=0. \end{equation*}
\end{lemma}

\begin{proof}
From their definitions we have
\begin{equation*}\mathcal{L}{\S}[f_\epsilon](x)= \frac{1}{2\pi}\int_{-\pi}^{\pi}(\sum_{j \in \Z}e^{-i\kappa j}f(\epsilon j))e^{i\kappa x}d\kappa = \frac{1}{2\pi}\int_{-\pi}^{\pi}(\sum_{j \in \Z}e^{-i\frac{\kappa}{\epsilon}\epsilon j}f(\epsilon j))e^{i\kappa x}d\kappa.\end{equation*}
Changing variables with $u = \kappa/\epsilon$
we get
\begin{equation*}\mathcal{L}{\S}[f_\epsilon](x)= \frac{1}{2\pi}\int_{-\pi/\epsilon}^{\pi/\epsilon}(\sum_{j \in \Z}\epsilon e^{-iu\epsilon j}f(\epsilon j))e^{iu\epsilon x}du.\end{equation*}
Exchanging the sum and integral and then computing the integral.
\begin{equation*}\mathcal{L}{\S}[f_\epsilon](x)=\sum_{j \in \Z} f(\epsilon j)\sinc(x-j).  \end{equation*}
This $\sinc$ is the normalized $\sinc$ function. 

Now put $\tilde{f}_{\ep}(X):=\mathcal{F}^{-1}[\theta_{\pi/\ep}\mathcal{F}[f]](X).$ $\tilde{f}_\ep$ is a band limited approximation of $f$. Using Plancherel's theorem 
\begin{equation*}\norm{f-\tilde{f}_\ep}^2_{L^2} = \norm{\mathcal{F}[f]-\theta_{\pi/\ep}\mathcal{F}[f]}^2_{L^2}=\int_{\kappa>|\pi/\ep|}\left \vert \mathcal{F}[f](\kappa)\right \vert^2d\kappa.\end{equation*}
Since $f \in H^s$ we have by Cauchy-Schwarz, when $\ep \in (0,1)$:
\bes\begin{split}
\norm{f-\tilde{f}_\ep}^2_{L^2} =&
\int_{\kappa>|\pi/\ep|}{1 \over |\kappa|^{2s}} |\kappa|^{2s} \left \vert \mathcal{F}[f](\kappa)\right \vert^2d\kappa\\
\le & \left( \int_{|\kappa| \ge \pi/\ep} |\kappa|^{-2s}d\kappa\right)^{1/2} 
\left( \int_{|\kappa| \ge \pi/\ep} |\kappa|^{2s} \left \vert \mathcal{F}[f](\kappa)\right \vert^2d\kappa\right)^{1/2}\\
\le & C \ep^{{s-1/2}}\|f\|_{H^s}.
\end{split}\ees
Since $s>1/2$ we see that $\lim_{\ep \to 0^+} \norm{f-\tilde{f}_\ep}_{L^2} = 0$.

Since $\tilde{f}_\ep$ is band limited, it is exactly equal to its cardinal series, see \cite{Stenger}, 
\begin{equation*}\tilde{f}_\ep(X)=\sum_{j \in \Z}f(\ep j)\sinc(X/\ep-j). \end{equation*} 
But this is exactly equal to $\L \S[f_\ep](X/\ep)$.
Therefore we have shown that \begin{equation*}\lim_{\epsilon \to 0^+}\norm{\mathcal{L}\mathcal{S}(f_\epsilon)(\cdot/\epsilon)-f(\cdot)}_{L^2} = 0. \end{equation*}
\end{proof}

Here is our first course-graining result:
\begin{theorem}
Fix $\Phi,\Psi \in H^3_{LIL}$ and $T_0>0$. Let $r$ and $p$ be the solution of \eqref{FORFPUT}
with initial data
$$
r(j,0) = \Phi(\ep j)/k(j) \mand p(j,0) = \Psi(\ep j).
$$
Put
$$
Q_\ep(X,\tau) = \L [ k r(\cdot,\tau/\ep)](X/\ep) \mand
P_\ep(X,\tau) = \L [ p(\cdot,\tau/\ep)](X/\ep).
$$
Suppose that $Q_0(X,\tau)$ and $P_0(X,\tau)$ solve \eqref{wave1} with initial
data $Q_0(X,0) = \Phi(X)$ and $P_0(X,0) = \Psi(X)$.
Then, almost surely,
$$
\lim_{\ep \to 0^+} \sup_{|\tau| \le T_0}\left( \|Q_\ep(X,\tau) - Q_0(X,\tau)\|_{L^2}
+\|P_\ep(X,\tau) - P_0(X,\tau)\|_{L^2}\right)=0.
$$
\end{theorem}

\begin{proof}
We show the limit for $\|P_\ep(X,\tau) - P_0(X,\tau)\|_{L^2}$ as the other is all but identical.
By the triangle inequality we have
$$
\|P_\ep(\cdot,\tau) - P_0(\cdot,\tau)\|_{L^2} \le 
\|P_\ep(\cdot,\tau) - \L\S[P_0(\ep \cdot,\tau)](\cdot/\ep)\|_{L^2}+\|\L\S[P_0(\ep \cdot,\tau)](\cdot/\ep) - P_0(\cdot,\tau)\|_{L^2}.
$$
The second term vanishes as $\ep \to 0^+$ by virtue of Lemma \ref{interp}.
(In fact, given \eqref{dalembert} one sees that this convergence happens uniformly for all $\tau \in \R$.)

For the first term we do a change of variables $X = \ep x$ and $\tau = \ep t$ to get
$$
\|P_\ep(\cdot,\tau) - \L\S[P_0(\ep \cdot,\tau)](\cdot/\ep)\|_{L^2}
= \sqrt{\ep} \|P_\ep(\ep \cdot,\ep t) - \L\S[P_0(\ep \cdot,\ep t)](\cdot)\|_{L^2}.
$$
Then we use the definition of $P_\ep$ and Lemma \ref{L is good} to get
$$
\|P_\ep(\cdot,\tau) - \L\S[P_0(\ep \cdot,\tau)](\cdot/\ep)\|_{L^2}
={1 \over 2\pi} \sqrt{\ep} \| p(\cdot,t) - \S[P_0(\ep \cdot,\ep t)]\|_{\ell^2}.
$$
Using \eqref{dalembert} and the formulas relating $\Phi$ and $\Psi$ to $A$ and $B$
in Theorem \ref{absolute theorem} we see 
$$
 \S[P_0(\ep \cdot,\ep t)](j)
 = {1 \over \sqrt{\tilde{k} \bar{m}}} \left( -A(\ep (j-ct) +B(\ep(j + ct)\right). 
$$
Thus we can use the final estimate in Theorem \ref{absolute theorem} to get
\begin{equation}
\label{last line}
\sup_{|\tau| \le T_0}
\|P_\ep(\cdot,\tau) - \L\S[P_0(\ep \cdot,\tau)](\cdot/\ep)\|_{L^2}
\le \sup_{|t|\le T_0/\ep}
{1 \over 2\pi} \sqrt{\ep} \| p(\cdot,t) - \S[P_0(\ep \cdot,\ep t)]\|_{\ell^2}
\le C_\om \sqrt{\ep \log \log(1/\ep)}.
\end{equation}
The right hand side goes to zero as $\ep \to 0^+$ and we are done.

\end{proof}
We have same result but the convergence is in mean:
\begin{theorem}
Fix $\Phi,\Psi \in H^4_{sr}$ and $T_0>0$. Let $r$ and $p$ be the solution of \eqref{FORFPUT}
with initial data
$$
r(j,0) = \Phi(\ep j)/k(j) \mand p(j,0) = \Psi(\ep j).
$$
Put
$$
Q_\ep(X,\tau) = \L [ k r(\cdot,\tau/\ep)](X/\ep) \mand
P_\ep(X,\tau) = \L [ p(\cdot,\tau/\ep)](X/\ep).
$$
Suppose that $Q_0(X,\tau)$ and $P_0(X,\tau)$ solve \eqref{wave1} with initial
data $Q_0(X,0) = \Phi(X)$ and $P_0(X,0) = \Psi(X)$.
Then
$$
\lim_{\ep \to 0^+} \E\left[\sup_{|\tau| \le T_0}\left( \|Q_\ep(X,\tau) - Q_0(X,\tau)\|_{L^2}
+\|P_\ep(X,\tau) - P_0(X,\tau)\|_{L^2}\right)\right]=0.
$$
\end{theorem}
\begin{proof}
As before, we start with the triangle inequality
\begin{equation*}
\begin{split}
&\E \left[ \sup_{|\tau| \leq T_0}\|P_\ep(\cdot,\tau) - P_0(\cdot,\tau)\|_{L^2} \right] \\\le &\E \left[ \sup_{|\tau| \leq T_0}
 \|P_\ep(\cdot,\tau) - \L\S[P_0(\ep \cdot,\tau)](\cdot/\ep)\|_{L^2} \right]+\E\left[\sup_{|\tau| \leq T_0}\|\L\S[P_0(\ep \cdot,\tau)](\cdot/\ep) - P_0(\cdot,\tau)\|_{L^2}\right].
\end{split}
\end{equation*}
The expected value does not see the second term, so it vanishes as $\epsilon \to 0^+$ by virtue of Lemma \ref{interp}. The same steps are valid up through \eqref{last line} only now we take expectation and use Theorem \ref{main mean theorem}
 \begin{equation*}
  \E\left[\sup_{|\tau| \le T_0}
  \|P_\ep(\cdot,\tau) - \L\S[P_0(\ep \cdot,\tau)](\cdot/\ep)\|_{L^2}\right]
 \le C\sqrt{\epsilon}\max\{\sigma_k, \sigma_m\}.
 \end{equation*}
 The right hand side vanishes as $\epsilon^+ \to 0$.
\end{proof}

\section{Simulations and Conclusion}
\label{conc}
We finish out the paper with supporting numerical simulations and a concluding discussion.
\subsection{Simulations}
We present various numerical data supporting our results. 
In our experiments, the springs $k$ are picked to be constant and the probability distribution of the masses $m$ s.t. $\bar{m}=1.$. We choose initial conditions 
\[r(j) =e^{-(\epsilon j)^2} \mand p(j)=-e^{-(\epsilon j)^2}. \]
From these
\[A(X)=e^{-X^{2}} \mand B(X) =0.\]
We numerically integrate \eqref{FORFPUT} to get $r(j,t)$ and use this to calculate the relative error which we call $\rho$
\[\rho\coloneqq \sup_{0 \leq t \leq T_0/\epsilon} \dfrac{\norm{(r(\cdot,t)-A(\epsilon(\cdot -ct))}_{\ell^2}}{\norm{r(t)}_{\ell^2}}. \]
According to Corollary \ref{relative cor}, for some $C_\omega$, $\rho$ will vanish to $0$ at least as fast as $C_\omega\sqrt{\epsilon\log\log(1/\epsilon)}.$ Seeing the $\sqrt{\log\log(\/\epsilon)}$ is numerically challenging and we make no claim that we do here. However, if it were to show up in the numerical calculations, it would be best to factor it out, so we calculate
\[\dfrac{\rho}{\sqrt{\log\log(1/\epsilon)}}. \]
Now this should vanish at a rate no slower than $C_\omega\sqrt{\epsilon}$, which on a log-log plot, should look like a straight line with a slope of $1/2$. Anything with a slope greater than $1/2$ is vanishing at a faster rate.

We move onto the figures after one aside on the methods of integration used. Since the total energy of the system is conserved, it is worth performing experiments with a symplectic integrator. A six-step version of Yoshida's method, see \cite{Yoshida}, was initially used, as well as the standard four-step Runge-Kutta method. As it turns out, these methods produce negligible differences for the time scales studied, so most of the experiments below all use only the four-step Runge-Kutta for the sake of computational efficiency.

Moving on, Figure \ref{fig:Relative Error for Fixed Random Masses} gives some numerical validations of our relative error results, since the slope produced by the log-log plot is greater than $1/2$. In this case, the realization of masses is the same for each $\epsilon$. Figure \ref{fig:40 Random Experiments} repeats the experiment in Figure \ref{fig:Relative Error for Fixed Random Masses} 40 times, displaying the results as a series of box plots. Figure \ref{fig:40 Random Experiments}, suggests the slope in Figure \ref{fig:Relative Error for Fixed Random Masses} is not a statistical anomaly.

It is worth noting that the most important tool of our analysis is $\chi_m$. For instance, it allows one to carry out similar analysis with many different kinds of sequences of masses. If the average of the masses exists and one knows the growth rate of $\chi_m$, then one can find an upper bound on the error. For example, in Figure \ref{fig:grows like root j}, we use a sequence of masses such that $\chi(j)$ grows like $\sqrt{j}$. In particular, using two types of masses $m_1$ and $m_2$, the following pattern works 
\begin{equation*}m_1,m_2,m_1,m_1,m_2,m_2,m_1,m_1,m_1,m_2,m_2,m_2... \end{equation*}
We conjecture without providing arguments that if $\chi_m$ grows like $|j|^p$, analysis would show that the relative error is bounded by an $\epsilon^{1-p}$ order term, which in this case coincides with the numerical results seen in Figure \ref{fig:grows like root j}. 

There are also hints in our work, see for example \ref{Res in Mean}, that for fixed $\epsilon$ and small $\sigma_m$, the mean of the error should be close to that of the system where masses and springs are taken to be constant average. Evidence for this is seen in Figure \ref{fig:Masses with small}. When $\sigma_m$ is smallest, then error from 40 trials, is concentrated around the error in the case of the system being constant coefficient, which was numerically calculated to be roughly $0.126$. In conclusion, the simulations are a strong affirmation of our analytic results and that our bounds are at least close to optimal.

\begin{figure}
\label{Fixed Random Masses}
\centering
\textbf{Relative Error for Fixed Random Masses}
    \par
    \centering
    \includegraphics[width=.75\textwidth,height=10cm]{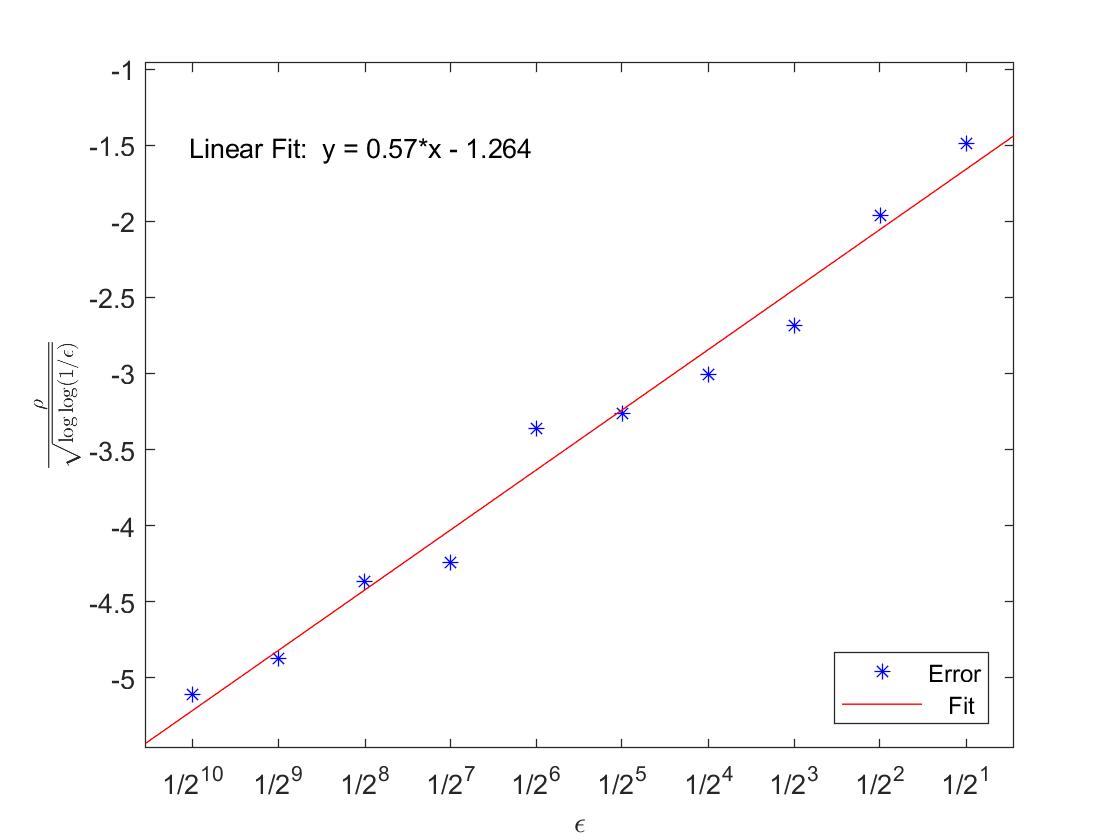}
    \caption{Figure \ref{fig:Relative Error for Fixed Random Masses} is a log-log plot of the relative error $\rho$ divided by $\sqrt{\log\log(1/\epsilon)}.$ }
    \label{fig:Relative Error for Fixed Random Masses}
\end{figure}

\begin{figure}
    \centering
    \textbf{40 Random Experiments}
    \par
    \includegraphics[width=.75\textwidth,height=10cm]{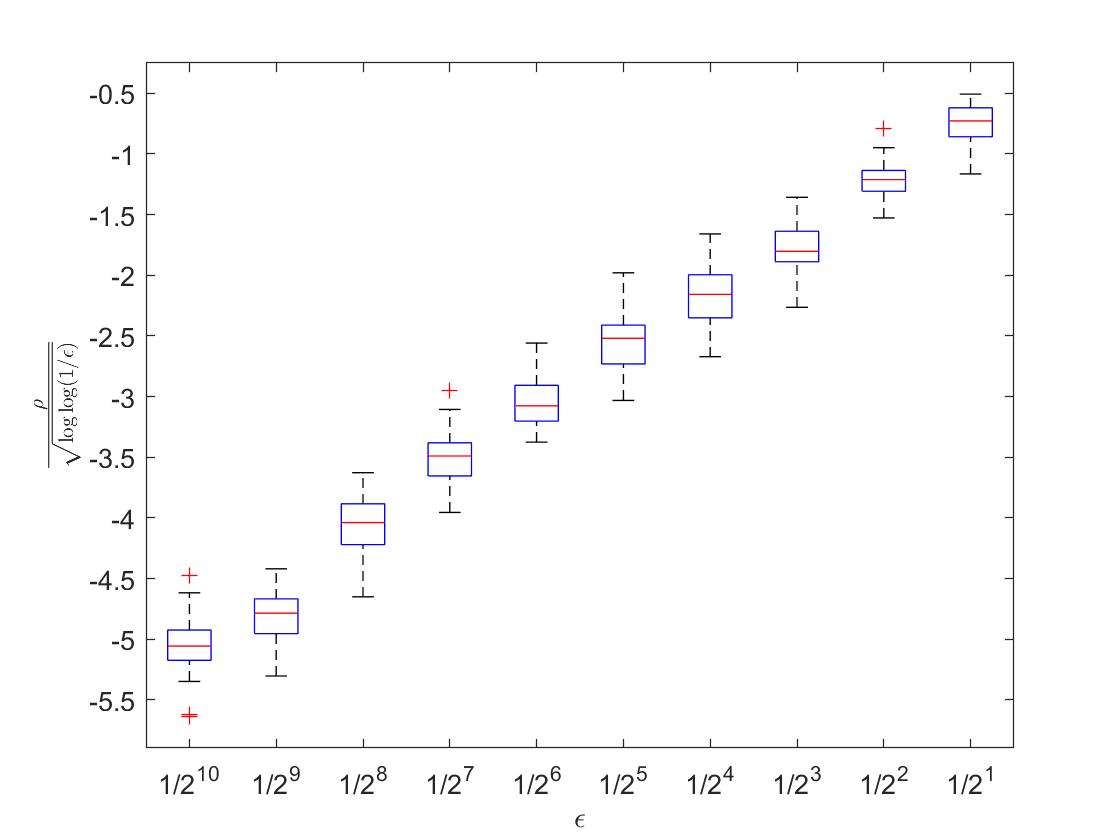}
    \caption{Figure \ref{fig:40 Random Experiments} is 10 box plots of 40 different realization of masses at 10 various epsilons. It is also log-log.}
    \label{fig:40 Random Experiments} 
\end{figure}

\begin{figure}
\label{Periodic Masses}
\centering
\textbf{Periodic Masses}
    \par
    \includegraphics[width=.75\textwidth,height=10cm]{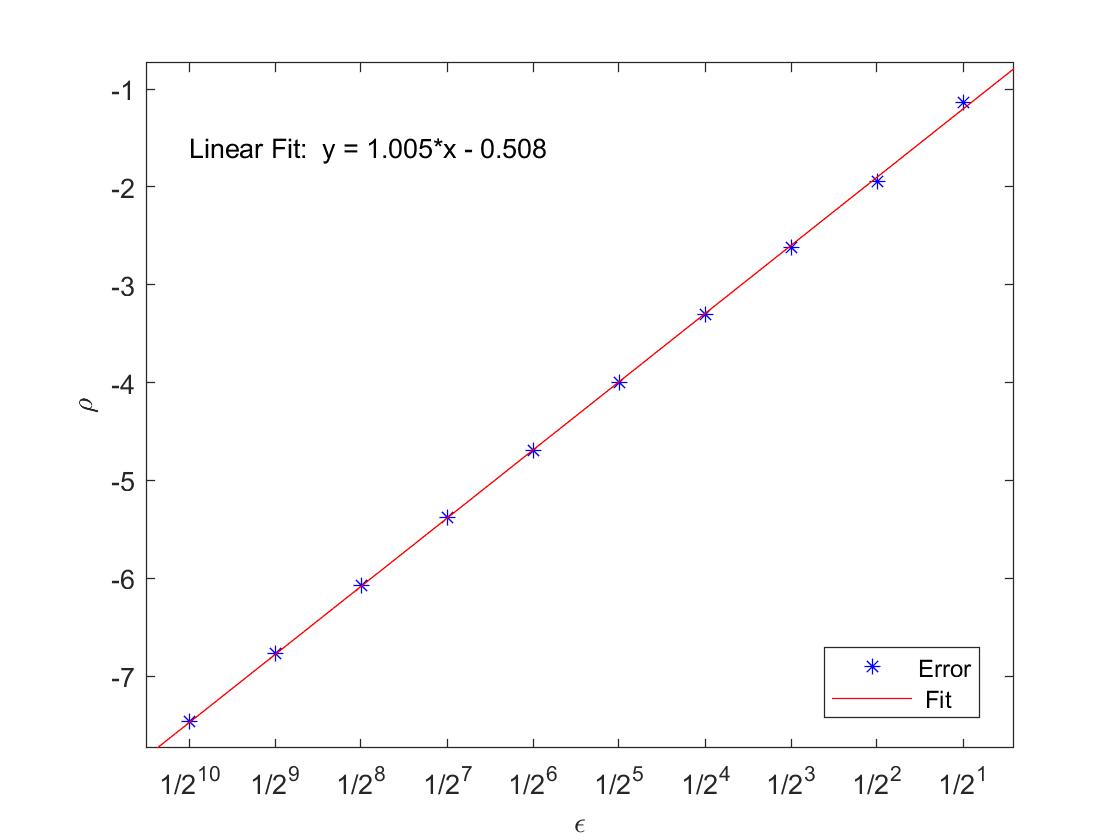}
    \caption{Figure \ref{fig:Periodic Masses} is a log-log plot of the relative error masses chosen periodically}
    \label{fig:Periodic Masses}
\end{figure}

\begin{figure}
\centering
\textbf{$\chi(j)$ Grows like $\sqrt{j}$} \par
    \includegraphics[width=.75\textwidth,height=10cm]{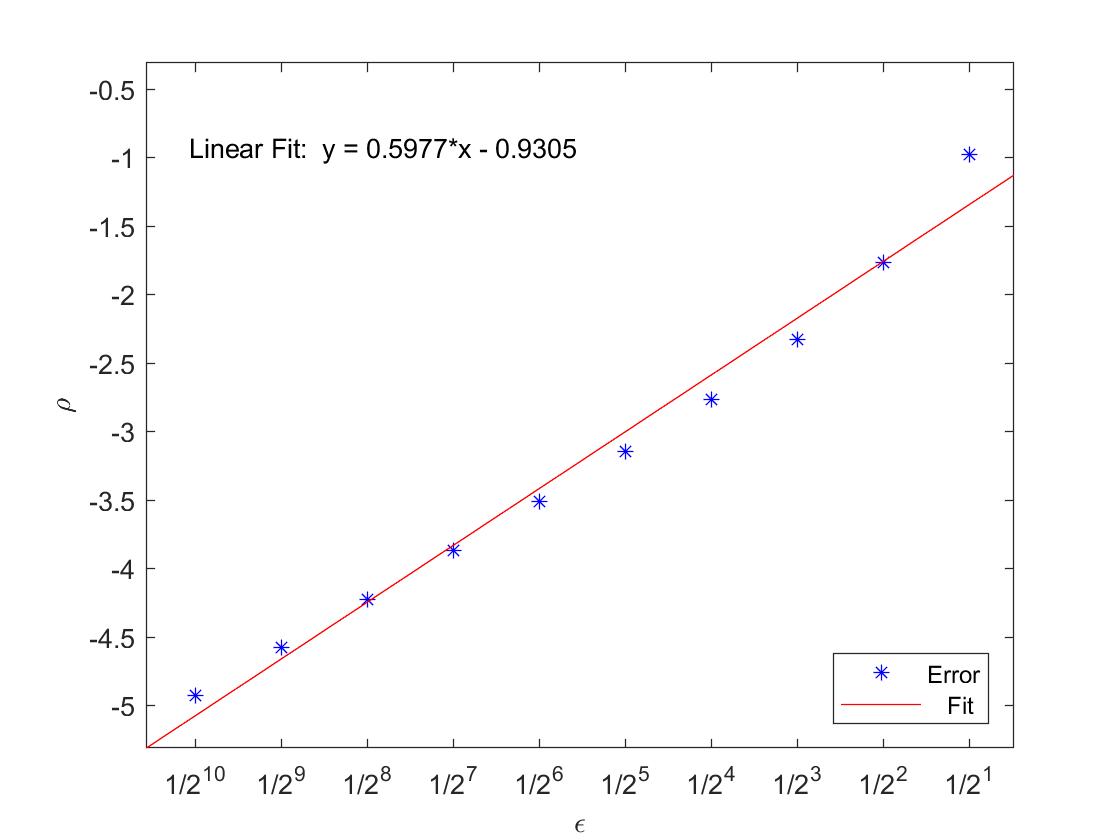}
    \caption{In Figure \ref{fig:grows like root j} masses are chosen so that $\chi(j)$ will grow like $\sqrt{j}$.}
    \label{fig:grows like root j}
\end{figure}

\begin{figure}
    \centering
    \textbf{Masses with Small $\sigma$}
    \par
    \includegraphics[width=.75\textwidth,height=10cm]{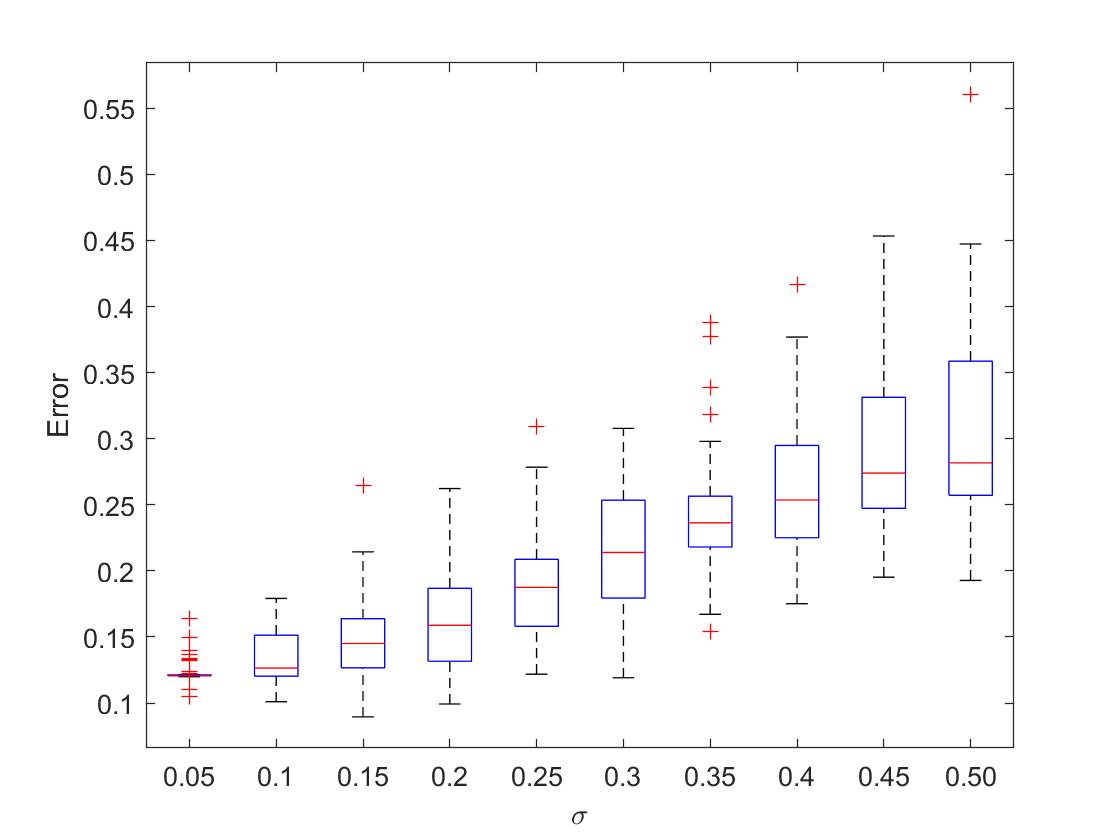}
    \caption{In Figure \ref{fig:Masses with small}  $\epsilon$ is fixed and small while $\sigma$ is varied and the absolute error is measured. When $\sigma$ is smallest, the data is concentrated near the error for the constant coefficient case. }
    \label{fig:Masses with small}
\end{figure}

\subsection{Conclusion} 
Our results are significant in several important ways regarding the description of approximate waves in the random polymer linear FPUT system. We have proven from first principles that that solutions to the wave equation are good approximate solutions to the system studied here. We showed that the absolute error only grows at most like $O(\log\log(1/\epsilon)$ almost surely and is constant in mean, but also small in mean if the masses and springs have small deviation. Using an interpolation operator with strong analytic properties we were able to show that the interpolated approximate solutions converged to interpolated true solutions in a relative sense a.s. and in expectation. Such results provide a rigorous justification for claiming that the relative error is made arbitrarily small by taking $\epsilon$ to be small. 

The advantage of our method comes from the use of the random walk in capturing the build up of error. Since random walks of independent variables are well studied and sharp asymptotic estimates are known, we were able to use the random walk to its full extent. Although it remains unproven if the error we achieved is sharp, the numerical results suggest it is close, and it seems nothing more about the asymptotics of the random walk, at least in the almost sure sense, could be used to prove sharper bounds. It also remains unclear if the random walk is an intrinsic part of the mechanics of the problem or if it is only a useful fiction for modeling the error. To what extent could it be further exploited here and in other models that have similar dynamics 

With this work we have laid the foundation for a couple of questions. First, can the error term be modeled by a random variable independent of $\epsilon$ with a nice probability distribution such as a Gaussian. There is also the question as to whether the results can be extended to higher dimensions. Probably most interesting, is what happens on larger time scales? In the periodic setting, solitons are known to last up to times proportional to $1/\epsilon^3$; however, it is not clear how to continue with the current methodology as was done to derive the KdV equations in the periodic case. This is mainly because one needs to make sense of $\lim_{n \to \infty}\sum_{|j| \leq n} \chi(j)/n$, which, even if one optimistically replaces $\chi(j)$ with $\sqrt{|j|}$, will diverge. This raises the question: is it is possible to find an effective equation describing the the dynamics for longer times and will these descriptions be statistical or is there room to achieve anything more definite, like the high probability and almost sure results constructed here?

\end{document}